\pdfoutput=1
\RequirePackage{ifpdf}
\ifpdf 
\documentclass[pdftex]{sigma}
\else
\documentclass{sigma}
\fi

\usepackage{mathtools}

\DeclarePairedDelimiter\floor{\lfloor}{\rfloor}

\newcommand{\set}[1]{\left\{ #1 \right\}}

\newcommand{\R}{{\mathbb R}}
\newcommand{\C}{{\mathbb C}}
\newcommand{\CQ}{{\mathcal{Q} }}

\newcommand{\CJk}{{\mathcal{J}_k}}
\newcommand{\CJn}{{\mathcal{J}_n}}
\newcommand{\CJ}{{\mathcal{J}}}

\newcommand{\GC}{G_{\mathbb{C}}}

\numberwithin{equation}{section}

\newtheorem{Theorem}{Theorem}[section]
\newtheorem{Corollary}[Theorem]{Corollary}
\newtheorem{Lemma}[Theorem]{Lemma}
\newtheorem{Proposition}[Theorem]{Proposition}
 { \theoremstyle{definition}
\newtheorem{Definition}[Theorem]{Definition}
\newtheorem{Remark}[Theorem]{Remark} }

\begin{document}

\allowdisplaybreaks

\newcommand{\arXivNumber}{1911.11842}

\renewcommand{\thefootnote}{}

\renewcommand{\PaperNumber}{041}

\FirstPageHeading

\ShortArticleName{Generalized B-Opers}

\ArticleName{Generalized B-Opers\footnote{This paper is a~contribution to the Special Issue on Integrability, Geometry, Moduli in honor of Motohico Mulase for his 65th birthday. The full collection is available at \href{https://www.emis.de/journals/SIGMA/Mulase.html}{https://www.emis.de/journals/SIGMA/Mulase.html}}}

\Author{Indranil BISWAS~$^\dag$, Laura P.~SCHAPOSNIK~$^\ddag$ and Mengxue YANG~$^\ddag$}

\AuthorNameForHeading{I.~Biswas, L.P.~Schaposnik and M.~Yang}

\Address{$^\dag$~Tata Institute of Fundamental Research, India}
\EmailD{\href{mailto:indranil@math.tifr.res.in}{indranil@math.tifr.res.in}}

\Address{$^\ddag$~University of Illinois at Chicago, USA}
\EmailD{\href{mailto:schapos@uic.edu}{schapos@uic.edu}, \href{mailto:myang59@uic.edu}{myang59@uic.edu}}

\ArticleDates{Received November 28, 2019, in final form May 02, 2020; Published online May 14, 2020}

\Abstract{Opers were introduced by Beilinson--Drinfeld [arXiv:math.AG/0501398]. In [\textit{J.~Math. Pures Appl.} \textbf{82} (2003), 1--42] a higher rank analog was considered, where the successive quotients of the oper filtration are allowed to have higher rank. We dedicate this paper to introducing and studying generalized $B$-opers (where ``$B$'' stands for ``bilinear''), obtained by endowing the underlying vector bundle with a~bilinear form which is compatible with both the filtration and the connection. In particular, we study the structure of these $B$-opers, by considering the relationship of these structures with jet bundles and with geometric structures on a~Riemann surface.}


\Classification{14H60; 31A35; 33C80; 53C07}

\rightline{\it In celebration of Motohico Mulase's 65th birthday.}

\renewcommand{\thefootnote}{\arabic{footnote}}
\setcounter{footnote}{0}

\section{Introduction}

The study of opers within geometry and mathematical physics has received much attention in the
last years, and in particular in connection with Higgs bundles in the recent years. In fact, certain opers
arise naturally as limits of Higgs bundles in the Hitchin components. Recall that the Higgs bundles,
corresponding to a complex Lie group~$\GC$, on a compact Riemann surface $\Sigma$ are given by the solutions
of Hitchin's equations:
\begin{gather}
F_A+ [\Phi, \Phi^*] = 0,\label{equation1}\\
\overline{\partial}_{A}\Phi = 0,
\label{equation2}
\end{gather}
where $F_A$ is the curvature of a unitary connection
$\nabla_A = \partial_{A}+\overline{\partial}_{A}$ associated to the Dolbeault operator
$\overline{\partial}_{A}$ on a principal $\GC$-bundle $P$ on $\Sigma$ and $\Phi$ is a $(1,0)$-form on
$\Sigma$ with values in the adjoint bundle $\operatorname{ad}(P)$~\cite{N1}. Given any solution $(A, \Phi)$ of~\eqref{equation1} and~\eqref{equation2}, there is a 1-parameter family of flat connections
\begin{gather}
\nabla_\xi = \xi^{-1}\Phi+\overline{\partial}_A+ \partial_A+\xi\Phi^{^*},\label{Laura1}
\end{gather}
parametrized by $\xi \in \C^\times = \C\setminus\{0\}$. Then, given a solution of
\eqref{equation1}--\eqref{equation2} in the ${\rm SL}(n, \C)$-Hitchin section, we can add a real parameter
$R > 0$ to~\eqref{Laura1} to obtain a natural family of connections with ${\rm SL}(n,\R)$ monodromy
\begin{gather*}
\nabla(\xi,R) := \xi^{-1}R\Phi+\overline{\partial}_A+ \partial_A+\xi R\Phi^{^*} .
\end{gather*}

In \cite{Gaiotto} Gaiotto conjectured that the space of opers should be obtained as the
$\hbar$-conformal limit of the Hitchin section: taking the limits $R \rightarrow 0$ and
$\xi\rightarrow 0$ simultaneously while holding the ratio $\hbar= \xi/R$ fixed. The
conjecture was recently established for general simple Lie groups by Dumitrescu, Fredrickson,
Kydonakis, Mazzeo, Mulase and Neitzke in~\cite{motohico1}, who also conjectured that this oper
is the {\it quantum curve} in the sense of Dumitrescu and Mulase~\cite{motohico2}, a
quantization of the spectral curve $S$ of the corresponding Higgs bundle by {\it topological
recursion}~\cite{motohico3}~-- see also references and details in \cite{Olivia}. Moreover,
very recently Collier and Wentworth showed in~\cite{collier2019conformal} that the above
conformal limit exists for spaces other than the Hitchin components.

\subsection*{Opers}

With views towards understanding other conformal limits of Higgs bundles and their appearance
through quantum curves, in this note we introduce {\it generalized $B$-opers} and begin a
program to study their geometry and topology, leaving for future work the extension of the
above results to this new setting. Opers were introduced by Beilinson--Drinfeld
\cite{BD, Opers1}; they were motivated by Drinfeld--Sokolov \cite{Opers2b, Opers2}.
Given a semisimple complex Lie group $G$, a $G$-oper on a compact Riemann surface
$\Sigma$ is
\begin{itemize}\itemsep=0pt
\item a holomorphic principal $G$-bundle $P$ on $\Sigma$ with a holomorphic connection $\nabla$,
and
\item a holomorphic reduction of structure group of $P$ to a Borel subgroup of $G$,
\end{itemize}
such that reduction satisfies the Griffiths transversality with respect to the connection
with the corresponding second fundamental forms being isomorphisms.
They admit reformulation in terms of holomorphic differential operators on a
Riemann surface. For example, ${\rm SL}(n, {\mathbb C})$-opers on $\Sigma$
are precisely holomorphic differential operators on $\Sigma$ of order $n$ and symbol $1$
with vanishing sub-principal term.

Opers also provide a coordinate-free description of the connections in the base space of a
generalized KdV hierarchy\footnote{We may think of a generalized KdV hierarchy to be an
infinite family of commuting flows on the space of connections on a principal $G$-bundle on a
noncompact curve.}. These connections can be translated to the usual $n$'th order differential
operators according to Drinfeld and Sokolov \cite{Opers2b, Opers2}. The study of opers has been
carried out from many different perspectives: as an example, Ben-Zvi and Frenkel studied how
opers arise from homogeneous spaces for loop groups and found a natural morphism between the
moduli spaces of spectral data and opers~\cite{MR1896178}. Moreover, Frenkel and Gaitsgory
studied opers on the formal punctured disc in~\cite{MR2263193} and showed that there is a
natural forgetful map from such opers to local systems on the punctured disc, which appears to
be of much importance for the geometric Langlands correspondence. Opers are also related to
the representations of affine Kac--Moody algebras at the critical level. In particular, from
\cite{MR2146349}, the algebra of functions on the space of $^LG$-opers, where $^LG$ is the
Langlands dual of $G$, on the formal disc is isomorphic to the classical $\mathcal{W}$-algebra associated to the Langlands dual Lie algebra $^L\mathfrak{g}$.

In recent years, new types of opers were introduced, such as $\mathfrak{g}$-opers\footnote{A~$\mathfrak{g}$-oper is a $G$-oper where $G$ is the group of automorphisms of $\mathfrak{g}$~\cite{Opers1}.} and Miura opers \cite{MR2146349}. We dedicate this paper to introducing and studying {\it generalized $B$-opers} (Definition \ref{def1}; where~``$B$'' stands for ``bilinear'' and not
``Borel'') which we show are closely related to geometric structures on the base Riemann surface
$\Sigma$ (Theorem~\ref{thm2}), and which encode certain classical $G$-opers. These are a
generalization of~\cite{Bi3}, a higher rank analog of the work in Beilinson--Drinfeld
\cite{BD, Opers1}, where the successive quotients of the oper filtration are allowed to have higher
rank. In our new setting, when the bilinear form is skew-symmetric and the rank of the bundle is~4,
it allows us to recover some of the~${\rm Sp}(4,\C)$-structure studied in~\cite{BA}.

The relationship between opers and geometric structures has been of much interests in the last half century, with the first results on this going back to work of R.C.~Gunning in \cite{Gu2}, where it was shown that there is a one-to-one correspondence between ${\rm SL}(n,\C)$-opers on a compact Riemann surface $\Sigma$ and complex projective structures on $\Sigma$. One should note that the term oper did not exist at the time, and hence Gunning calls his family of flat bundles associated to the complex projective structures {\it indigenous bundles} on the Riemann surface. Much work has been done sine then on geometric structures arising through opers, and the reader may want to refer to~\cite{sanders18} and references therein for further details.

In \cite{Bi3}, a more general class of opers was studied, where $\operatorname{rank}(E) = nr$ and
the rank of each successive quotient $E_i/E_{i-1}$ is $r$ (the above two conditions remain
unchanged). In the present paper, we incorporate a non-degenerate bilinear form $B$ and require the
(not necessarily full) filtration and the connection appearing in a $G$-oper to be compatible
with it, thus leading to the natural objects which we call {\it generalized $B$-opers}, and for which we consider geometric structures arising through them in our main Theorem~\ref{thm2}.

The paper is organized as follows. We first review some of the basic definitions and properties of classical opers, following~\cite{Opers1} and~\cite{MR1896178} (done in Section~\ref{review}), and then introduce in Section~\ref{Boper} what we call {\it generalized $B$-opers} and their main properties. These are triples $(E,{\mathcal F}, D)$, where~$\mathcal F$ is a $B$-filtration of a holomorphic vector bundle $E$ associated to a~non-degenerate bilinear form~$B$ (as in Definition~\ref{parity}) and~$D$ is a $B$-connection on~$E$ (as in Definition~\ref{Bcon}), such that for every $1 \leq i \leq n-1$ the condition $D(E_i) \subset E_{i+1}\otimes K_\Sigma$ holds, and the resulting homomorphism
\begin{gather*}
E_i/E_{i-1} \longrightarrow (E_{i+1}/E_i)\otimes K_\Sigma
\end{gather*}
(constructed in \eqref{e3}) is an isomorphism. Generalized $B$-opers are closely related to jets, and we investigate this correspondence in Section~\ref{jets}, where in particular we show in Theorem~\ref{prop3} and Proposition~\ref{lem4} that there is a correspondence between the generalized $B$-opers $(E,{\mathcal F}, D)$ on~$\Sigma$ and triples consisting of the following:
 \begin{itemize}\itemsep=0pt
\item[(1)] a fiberwise non-degenerate symmetric bilinear form ${\mathbb B}_n$ on
$\CQ ^* := \big(E/E_{n-1}\otimes K_\Sigma^{(n-1)/2}\big)^*$,
\item[(2)] a holomorphic connection on $\CQ^*$ that preserves this bilinear form ${\mathbb B}_n$, and
\item[(3)] a classical ${\rm Sp}(n, {\mathbb C})$-oper for $n$ even and a classical ${\rm SO}(n, {\mathbb C})$-oper for $n$ odd.
\end{itemize}

The above correspondence of spaces can be taken further, which is done in Section~\ref{construction}. More precisely, the known relation between classical opers and projective structures can be extended to this generalized setting, leading to one of our main results:

\medskip

\noindent {\bf Theorem \ref{thm2}.} {\it For integers $n \geq 2$, $n \not= 3$ and $r \ge 1$, the
space of all generalized $B$-opers of filtration length $n$ and $\mbox{rank}(E_i/E_{i-1})=r$ is in correspondence with
\[
{\mathcal C}_\Sigma\times {\mathfrak P}(\Sigma)\times \left(\bigoplus_{i=2}^{[n/2]} H^0\big(\Sigma, K^{\otimes 2i}_\Sigma\big)\right) , \]
where ${\mathcal C}_\Sigma$ denotes the space of all flat orthogonal bundles of rank $r$ on $\Sigma$, which is independent of $i$, and
${\mathfrak P}(\Sigma)$ is the space of all projective structures on $\Sigma$. }

\medskip

We conclude the paper describing certain naturally defined Higgs bundles appearing through
generalized $B$-opers which carry certain Lagrangian structure. In Section~\ref{sec:higgs}
we initiate the study of those Higgs bundles, leaving to future work their further study.

\section[Generalized $B$-opers]{Generalized $\boldsymbol{B}$-opers}

In what follows we shall first give a brief overview of classical $G$-opers for $G$ a
connected reductive complex Lie group (this is done in Section~\ref{intro}), and then introduce our
generalization, what we call ``Generalized $B$-opers'', in Section~\ref{Boper}.

Let $\Sigma$ be a compact connected Riemann surface of genus $g$; its holomorphic cotangent bundle will be
denoted by $K_\Sigma$.

\subsection[$G$-Opers]{$\boldsymbol{G}$-Opers}\label{review}\label{intro}

We shall recall here the following definitions from Beilinson and Drinfeld's paper~\cite{Opers1}. Fix a~Borel subgroup $B \subset G$ together with a maximal torus $H \subset B$. Let $\mathfrak{h} \subset
\mathfrak{b} \subset \mathfrak{g}$ be the Lie algebras of $H \subset B \subset G$ respectively. Fix a set of
positive simple roots $\Gamma \subset \mathfrak{h}^*$ with respect to~$\mathfrak{b}$. For each~$\alpha$ in the dual space $\mathfrak{h}^*$, set
\begin{gather*}
\mathfrak{g}^\alpha := \set{ x \in \mathfrak{g} \,|\, [a, x] = \alpha(a)x \ \forall\,a \in \mathfrak{h} }.
\end{gather*}
Then, there is a unique Lie algebra grading $\mathfrak{g} = \bigoplus_k \mathfrak{g}_k$ such that the following holds
\begin{gather*}
\mathfrak{g}_0 = \mathfrak{h}, \qquad \mathfrak{g}_1 = \bigoplus_{\alpha \in \Gamma} \mathfrak{g}^\alpha, \qquad \mathfrak{g}_{-1} = \bigoplus_{\alpha \in \Gamma} \mathfrak{g}^{-\alpha}.
\end{gather*}

Let $P$ be a holomorphic principal $B$-bundle on $\Sigma$, and let $\mathcal{E}_P$ be the Lie algebroid over $\Sigma$ of infinitesimal symmetries of $P$. This $\mathcal{E}_P$ is same as the Atiyah bundle for $P$ \cite{At11}. Let $\mathcal{E}_P^{\mathfrak{g}}$ be the Lie algebroid over $\Sigma$ of infinitesimal symmetries of the holomorphic principal $G$-bundle $Q
 = P \times_B G$ obtained by extending the structure group of~$P$ to~$G$.
Let $\mathfrak{b}_P := P(\mathfrak{b})$ (respectively, $\mathfrak{g}_P := P(\mathfrak{g})$)
be the holomorphic vector bundle over $\Sigma$ associated to $P$ for the $B$-module $\mathfrak{b}$ (respectively,
$\mathfrak{g}$); so $\mathfrak{b}_P$ is the adjoint bundle $\operatorname{ad}(P)$.
Then, there is a filtration $\mathfrak{g}_P^k \subset \mathfrak{g}_P$, and we shall let $\mathcal{E}_P^k$ be the preimage of $\mathfrak{g}_P^k / \mathfrak{b}_P \subset \mathfrak{g}_P / \mathfrak{b}_P = \mathcal{E}_P^{\mathfrak{g}} / \mathcal{E}_P$. One can then define another filtration $\mathcal{E}_P^k \subset \mathcal{E}_P^{\mathfrak{g}}$, $k \le 0$.
Considering $\mathfrak{g}_P^{-\alpha}$ the $P$-twist of the $B$-module $\mathfrak{g}^{-\alpha}$, one has that $\mathcal{E}_P^{-1} / \mathcal{E}_P = \mathfrak{g}_P^{-1} / \mathfrak{b}_P = \bigoplus_{\alpha \in \Gamma} \mathfrak{g}_P^{-\alpha}$.

\begin{Definition}[{$G$-oper}]\label{operdef} A $G$-oper on $\Sigma$ is a holomorphic principal $B$-bundle $P$ on $\Sigma$ with a holomorphic connection $\omega\colon T\Sigma \longrightarrow \mathcal{E}_P^{\mathfrak{g}}$ such that $\omega(T\Sigma) \subset \mathcal{E}_P^{-1}$, and for every $\alpha \in \Gamma$, the following composition of homomorphisms is an isomorphism
\begin{gather*}
T\Sigma \stackrel{\omega}{\longrightarrow} \mathcal{E}_P^{-1} \longrightarrow \mathcal{E}_P^{-1} / \mathcal{E}_P \longrightarrow \mathfrak{g}_P^{-\alpha} .
\end{gather*}
\end{Definition}

In order to draw a parallel between $G$-opers and our generalization, we shall focus now on ${\rm SL}(n,\C)$-opers, which can be described as follows from~\cite{Opers1}. Indeed, opers are some vector bundles satisfying certain conditions that give them an oper structure.

\begin{Definition}\label{defGLnoper} A ${\rm GL}(n,\C)$-oper on $\Sigma$ is a~holomorphic vector bundle~$E$ of rank~$n$ with a~complete filtration of holomorphic subbundles
\begin{gather*}
0 = E_0 \subsetneq E_1 \subsetneq E_2 \subsetneq \cdots \subsetneq E_n = E
\end{gather*}
and a holomorphic connection $D\colon E \to E \otimes K_\Sigma$ such that
\begin{enumerate}\itemsep=0pt
\item[1)] (the filtration is a complete flag) the quotient $E_i / E_{i-1}$ is a line bundle for all $1 \le i \le n$;
\item[2)] (Griffiths' transversality) $D(E_i) \subset E_{i+1} \otimes K_\Sigma$ for $1 \le i \le n-1$;
\item[3)] (non-degeneracy) and $D$ induces an isomorphism
\begin{gather*}
E_i / E_{i-1} \xrightarrow{\sim} (E_{i+1} / E_i) \otimes K_\Sigma .
\end{gather*}
\end{enumerate}
We denote this ${\rm GL}(n,\C)$-oper by $(E, \set{E_i}, D)$.
\end{Definition}

\begin{Remark}\label{rem1} An ${\rm SL}(n,\C)$-oper is a ${\rm GL}(n,\C)$-oper $(E, \set{E_i}, D)$ on $\Sigma$ such that the determinant line bundle $\bigwedge^n E$ is holomorphically trivial and the connection on $\bigwedge^n E$ induced by $D$ coincides with the trivial connection on ${\mathcal O}_\Sigma$. An ${\rm Sp}(2n,\C)$-oper is an ${\rm SL}(2n,\C)$-oper $(E, \set{E_i}, D)$ with a horizontal symplectic form on $E$ such that $E_i^\perp = E_{n-i}$. An ${\rm SO}(n,\C)$-oper is an ${\rm SL}(n)$-oper $(E, \set{E_i}, D)$ with a~horizontal non-degenerate symmetric bilinear form $B$ on $E$ such that $E_i^\perp = E_{n-i}$.
\end{Remark}

\subsection[Second fundamental form and generalized $B$-opers]{Second fundamental form and generalized $\boldsymbol{B}$-opers}\label{Boper}

In order to generalize the above definitions to account for the structure appearing through bilinear forms and their preserved filtrations, consider a pair $(E, B)$, where~$E$ is a holomorphic vector bundle over $\Sigma$, and
\begin{gather}\label{1}
B \colon \ E\otimes E \longrightarrow {\mathcal O}_\Sigma
\end{gather}
is a holomorphic homomorphism such that for every point $x \in \Sigma$, the bilinear
form
\[
B(x) \colon \ E_x\otimes E_x \longrightarrow \mathbb C
\]
is non-degenerate, meaning for every point $x \in \Sigma$ and each nonzero vector
$v \in E_x$ there is some $w \in E_x$ such that $B(x)(v, w) \not= 0$.

\begin{Definition}The form $B$ in \eqref{1} is called {\em symplectic} (respectively, {\em orthogonal})
if it satisfies $B(x)(v, w)=- B(x)(w, v)$ (respectively, $B(x)(v, w)= B(x)(w, v)$) for all
$x \in \Sigma$ and for all $v, w \in E_x$.
\end{Definition}

The bilinear forms we shall consider in the present paper will be either symplectic or orthogonal.
Let
\begin{gather*}p_0 \colon \ E \longrightarrow \Sigma\end{gather*} be the natural projection.
For any holomorphic subbundle $F \subset E$, define the subbundle
\begin{gather*}
F^\perp := \{w \in E \,|\, B(p_0(w))(w, v) = 0\ \forall\,v \in F_{p_0(w)}\} \subset E .
\end{gather*}

Note that since $B(p_0(w))(w, v) = 0$ if and only if $B(p_0(w))(v, w) = 0$, the subbundle $F^\perp$ does not change if $B(p_0(w))(w, v)$ in the definition of $F^\perp$ is replaced by $B(p_0(w))(v, w)$.

\begin{Lemma}\label{lem-1} The $C^\infty$ subbundle $F^\perp \subset E$ is a holomorphic subbundle. Moreover, $F^\perp$ is canonically isomorphic to the dual bundle $(E/F)^*$. Also, $\big(F^\perp\big)^\perp = F$.
\end{Lemma}

\begin{proof}Let $\overline{\partial}_E \colon E \longrightarrow E\otimes \Omega^{0,1}_\Sigma = E\otimes \overline{K}_\Sigma$ be the Dolbeault operator defining
the holomorphic structure of the holomorphic
vector bundle~$E$. The given condition that the bilinear form $B$ is holomorphic implies
that
\begin{gather}\label{bi}
\overline{\partial}B(s_1, s_2) = B\big(\overline{\partial}_E(s_1), s_2\big)+ B\big(s_1,\overline{\partial}_E(s_2)\big)
\end{gather}
for all locally defined $C^\infty$ sections $s_1$ and $s_2$ of $E$. Since $F$ is a holomorphic
subbundle of $E$, it follows that $\overline{\partial}(s)$ is a $C^\infty$ (local) section of
$F\otimes\Omega^{0,1}_\Sigma$ for every (local) $C^\infty$ section~$s$ of~$F$. Therefore if~$s_1$ is a~$C^\infty$ locally defined section of $F^\perp$ and $s_2$ is a $C^\infty$ locally defined section of~$F$, then from~\eqref{bi} it follows immediately that
\[
B\big(\overline{\partial}_E(s_1), s_2\big) = 0 .
\]
Now, this implies that $\overline{\partial}_E(s_1)$ is a locally defined section of~$F^\perp\otimes\Omega^{0,1}_\Sigma$. Consequently, $F^\perp$ is a~holomorphic subbundle on~$E$.

The bilinear form $B$ identifies $E$ with its dual $E^*$. Consider the
following composition of homomorphisms
\[
F^\perp \hookrightarrow E \stackrel{\sim}{\longrightarrow} E^* ,
\]
which is injective. The image of $F^\perp$ in $E^*$ evidently coincides with the image of the natural injective homomorphism
\begin{gather*}
(E/F)^* \longrightarrow E^* .
\end{gather*}
Consequently, $F^\perp$ is identified with $(E/F)^*$.

Finally, from the definition of $F^\perp$ it follows immediately that $(F^\perp)^\perp = F$.
\end{proof}

In the case of Hermitian holomorphic vector bundles, Kobayashi showed in \cite{MR3643615} that
$F^\perp$ is not a holomorphic subbundle if $F$ is not preserved by the
Chern connection on $E$ (see Definition~\ref{def2}).

\begin{Definition}\label{parity} A $B$-\textit{filtration} of a holomorphic vector bundle~$E$ is an increasing filtration $\mathcal{F}$ of holomorphic subbundles
\begin{gather}\label{e4}
0 = E_0 \subsetneq E_1 \subsetneq E_2 \subsetneq \cdots \subsetneq E_{n-1} \subsetneq E_{n-1} \subsetneq E_{n} = E
\end{gather}
for which the following two conditions hold:
\begin{enumerate}\itemsep=0pt
\item[1)] the length $n\in \mathbb N$ of the filtration is even when $B$ is symplectic, and
it is odd when $B$ is orthogonal;
\item[2)] $E^\perp_i = E_{n-i}$ for all $1 \leq i \leq n-1$.
\end{enumerate}
\end{Definition}

Holomorphic connections where introduced by Atiyah in~\cite{At11} (see also~\cite{at}). We recall that a~holomorphic connection on a~holomorphic vector bundle~$W$ on~$\Sigma$ is a~holomorphic differential operator
\begin{gather*}
D \colon \ W \longrightarrow W\otimes K_\Sigma
\end{gather*}
such that
\begin{gather}\label{e2}
D(fs) = fD(s)+s\otimes {\rm d}f
\end{gather}
for all locally defined holomorphic sections $s$ of $W$ and all locally defined holomorphic functions~$f$ on~$\Sigma$.

Note that the Leibniz condition in \eqref{e2} implies that the order of the differential
operator~$D$ is one. Moreover, a holomorphic connection on $\Sigma$ is automatically flat
because the sheaf of holomorphic two-forms on $\Sigma$ is the zero sheaf. Through this, we can
impose the following compatibility condition with respect to the bilinear form~$B$.

\begin{Definition}\label{Bcon} Let $E$ be a holomorphic vector bundle on $\Sigma$ equipped with a bilinear form~$B$. A~$B$-\textit{connection} on~$E$ is a holomorphic connection $D$ on $E$ such that
\begin{gather}\label{bc}
\partial (B(s, t)) = B(D(s), t) + B(s, D(t))
\end{gather}
for all locally defined holomorphic sections $s$ and $t$ of~$E$. (Note that $\partial (B(s, t)) = {\rm d}(B(s, t))$ because $B(s, t)$ is a holomorphic function.)
\end{Definition}

Given a holomorphic subbundle $F \subset E$ and the corresponding quotient map
\begin{gather*}q_F \colon \ E \longrightarrow E/F ,\end{gather*} it follows from \eqref{e2} that the composition
of homomorphisms
\begin{gather}\label{2}
F \hookrightarrow E \stackrel{D}{\longrightarrow} E\otimes K_\Sigma
\stackrel{q_F\otimes {\rm Id}_{K_\Sigma}}{\longrightarrow} (E/F)\otimes K_\Sigma ,
\end{gather}
is in fact ${\mathcal O}_\Sigma$-linear.

\begin{Definition}\label{def2b}
The holomorphic section
\begin{gather*}
S(D, F) \in H^0(\Sigma, \operatorname{Hom}(F, E/F)\otimes K_\Sigma)
\end{gather*}
given by the composition of homomorphism in \eqref{2} is called the {\em second fundamental form} of $F$ for the connection~$D$.
\end{Definition}

The above construction of the second fundamental form in Definition~\ref{def2b} can be generalized from subbundles of $E$ to filtration of subbundles of~$E$ as follows.

\begin{Definition}\label{def2} Let $D$ be a holomorphic connection on $E$, and let
\[
F_1 \subset F_2 \subset E\qquad \text{and}\qquad F_3 \subset F_4 \subset E
\]
be holomorphic subbundles such that \begin{gather*}D(F_1) \subset F_3\otimes K_\Sigma\qquad \text{and}\qquad
D(F_2) \subset F_4\otimes K_\Sigma .\end{gather*} Then, the {\rm second fundamental form} of
$(F_1,F_2,F_3,F_4)$ for the connection $D$ is the map
\begin{align}\label{e3}
S(D; F_1,F_2,F_3, F_4)\colon \ F_2/F_1 & \longrightarrow (F_4/F_3)\otimes K_\Sigma,\\
s& \longmapsto D(\widetilde{s}),\nonumber
\end{align}
that sends any locally defined holomorphic section $s$ of $F_2/F_1$ to the image
of $D(\widetilde{s})$ in $(F_4/F_3)\otimes K_\Sigma$, where $\widetilde{s}$ is any
locally defined holomorphic section of the subbundle $F_2$ that projects to $s$.
\end{Definition}

In view of \eqref{e2}, note that the above conditions in Definition \ref{def2} imply that
$F_1 \subset F_3$ and $F_2 \subset F_4$.

\begin{Lemma}The section $S(D; F_1,F_2,F_3, F_4)(s)$ in~\eqref{e3} is independent of the lift $\widetilde{s}$ of $s$ to~$F_2$.
\end{Lemma}

\begin{proof}This follows from the condition $D(F_1) \subset F_3\otimes K_\Sigma$. Indeed,
for~$s$ a local section of~$F_2/F_1$, let $\widetilde{s}_1$, $\widetilde{s}_2$ be two lifts of $s$
to~$F_2$. This means $\widetilde{s}_1-\widetilde{s}_2$ is a local section of~$F_1$. Therefore,
from the given condition that $D(F_1) \subset F_3\otimes K_\Sigma$ we know that
\begin{gather*}
D(\widetilde{s}_1)-D(\widetilde{s}_2) = D(\widetilde{s}_1-\widetilde{s}_2)
\end{gather*}
is a local section of $F_3\otimes K_\Sigma$. Consequently, the section $S(D; F_1,F_2,F_3, F_4)(s)$ is independent of the lift $\widetilde{s}$ of $s$ to $F_2$.
\end{proof}

From \eqref{e2} it follows immediately that for $f$ a locally defined holomorphic function on~$\Sigma$ and~$s$ a locally defined holomorphic section of $F_2/F_1$ to~$F_2$, then the equality
\begin{gather*}
S(D; F_1,F_2,F_3, F_4)(fs) = f\cdot S(D; F_1,F_2,F_3, F_4)(s)
\end{gather*}
holds. In other words, $S(D; F_1,F_2,F_3, F_4)$ is an ${\mathcal O}_\Sigma$-linear homomorphism of coherent analytic sheaves.

Recall from Section \ref{review} that a ${\rm GL}(n,{\mathbb C})$-oper on $\Sigma$ is a~${\rm GL}(n,{\mathbb C})$-local system (that is a~holomorphic vector bundle~$E$ with a~holomorphic connection) equipped with a complete flag satisfying some extra conditions~(2) and~(3) in Definition~\ref{defGLnoper}. The following definition generalizes these two conditions.

\begin{Definition}\label{def1}A generalized $B$-{\it oper} is a triple $(E,{\mathcal F}, D)$, where $\mathcal F$ is a $B$-filtration
\[
0 = E_0 \subsetneq E_1 \subsetneq E_2 \subsetneq \cdots \subsetneq E_{n-2} \subsetneq E_{n-1} \subsetneq E_{n} = E
\]
as in \eqref{e4} and $D$ is a $B$-connection on $E$, such that
\begin{enumerate}\itemsep=0pt
\item[1)] $D(E_i) \subset E_{i+1}\otimes K_\Sigma$ for all $1 \leq i \leq n-1$, and
\item[2)] for every $1 \leq i \leq n-1$, the homomorphism
\begin{gather}\label{e5}
S(D; E_{i-1}, E_i,E_i, E_{i+1})\colon \ E_i/E_{i-1} \longrightarrow (E_{i+1}/E_i)\otimes K_\Sigma
\end{gather}
(constructed in \eqref{e3}) is an isomorphism. For simplicity of notation, we shall
sometimes write $S_i(D)$ for the above map $S(D; E_{i-1}, E_i,E_i, E_{i+1})$.
\end{enumerate}
\end{Definition}

\begin{Remark} Note that properties (1) and (2) from Definition \ref{def1} are natural counterparts to (2)
and~(3) in Definition~\ref{defGLnoper}. Moreover, the above definition can be seen as a
generalization of the opers studied in~\cite{Bi3}. Here we have the additional
structure of a bilinear form~$B$, and it fits within the broader picture of $(G,P)$-opers
which is introduced in the work of Collier and Sanders in~\cite{BA}.
\end{Remark}

\subsection[The underlying bundle of a generalized $B$-oper]{The underlying bundle of a generalized $\boldsymbol{B}$-oper}\label{sym11}

Let $(E, {\mathcal F}, D)$ be a~generalized $B$-oper as in Definition \ref{def1} above. In what follows we shall construct a~natural $K^{\otimes(n-1)}_\Sigma$-valued symmetric form ${\mathbb S}'$ on the subbundle
$E_1$ in Definition \ref{def1}. For this, first note that for any $1 \leq i \leq n-1$, the isomorphism
\[
S(D; E_{i-1}, E_i,E_i, E_{i+1}) \colon \ E_i/E_{i-1} \longrightarrow (E_{i+1}/E_i)\otimes K_\Sigma
\]
in~\eqref{e5} tensored with the identity map of $K^{\otimes(i-1)}_\Sigma$ produces isomorphisms $S_i$
defined by
\begin{gather}
S_i := S(D; E_{i-1}, E_i,E_i, E_{i+1})\otimes {\rm Id}_{K^{\otimes(i-1)}_\Sigma}\colon \nonumber\\
\qquad {} (E_i/E_{i-1})\otimes K^{\otimes (i-1)}_\Sigma
\longrightarrow (E_{i+1}/E_i)\otimes K_\Sigma
\otimes K^{\otimes (i-1)}_\Sigma = (E_{i+1}/E_i)\otimes K^{\otimes i}_\Sigma .\label{sym1}
\end{gather}
By composing the above maps from~\eqref{sym1} we obtain an isomorphism
\begin{gather}\label{e6}
{\mathbb S}' := S_{n-1}\circ\cdots\circ S_1\colon \ E_1 \longrightarrow (E_{n}/E_{n-1})\otimes K^{\otimes(n-1)}_\Sigma .
\end{gather}

\begin{Lemma}\label{lem11}The bilinear form $B$ produces the following isomorphism on a $B$-filtration as in
Definition~{\rm \ref{parity}}
\[
E^*_{n-i} \stackrel{\sim}{\longrightarrow} E_n/E_i
\]
for every $1 \leq i \leq n-1$.
\end{Lemma}

\begin{proof}
Since $E^\perp_i = E_{n-i}$ (see Definition \ref{parity}), from Lemma \ref{lem-1} we
conclude that $(E/E_i)^* = E^\perp_i = E_{n-i}$. The lemma follows from this.
\end{proof}

Using the identification between $E^*_1$ and $E_n/E_{n-1}$ in Lemma~\ref{lem11}, the isomorphism ${\mathbb S}'$ in~\eqref{e6} can be seen as a holomorphic
section
\begin{gather}\label{e7}
{\mathbb S}' \in H^0\big(\Sigma, (E^*_1)^{\otimes 2}\otimes K^{\otimes(n-1)}_\Sigma\big) =
H^0\big(\Sigma, (E_n/E_{n-1})^{\otimes 2}\otimes K^{\otimes(n-1)}_\Sigma\big).
\end{gather}

\begin{Lemma}\label{lem12}Take a holomorphic vector field $v \in H^0(U, T\Sigma)$ on an open subset
$U \subset \Sigma$, and define the map
\begin{align}
V \colon \ H^0(U, E) &\longrightarrow H^0(U, E),\nonumber\\
s & \longmapsto D(s)(v) ,\label{mapaV}
\end{align}
where as usual, $H^0(U, E)$ denotes the space of all holomorphic sections of~$E$ over~$U$. The map~$V$ from~\eqref{mapaV} satisfies the identity
\[
B(V(s), t) + B(s, V(t)) = v(B(s, t)) .
\]
\end{Lemma}

\begin{proof} This can be proved using~\eqref{bc}. Indeed, since~$D$ is a $B$-connection, the following equality of $1$-forms on $\Sigma$ holds
\begin{gather}\label{eq:DBconn}
{\rm d}(B(s,t)) = B(D(s),t) + B(s,D(t))
\end{gather}
(see \eqref{bc}), and thus contracting both sides of~\eqref{eq:DBconn} by the vector field~$v$ the identity in the lemma is obtained.
\end{proof}

\begin{Lemma} \label{lem22} Let $(E, {\mathcal F}, D)$ be a generalized $B$-oper as in Definition~{\rm \ref{def1}}. Take $s, t \!\in\! H^0(U, E_1)$. Then,
\begin{gather}\label{e8}
B\big(V^{n-1}(s), t\big) + (-1)^{n-2} B\big(s, V^{n-1}(t)\big) = 0
\end{gather}
for the map $V$ in~\eqref{mapaV}.
\end{Lemma}

\begin{proof}Through Lemma \ref{lem12}, for $s, t \in H^0(U, E_1) \subset H^0(U, E)$ we have that
\begin{gather}\label{b1}
B\big(V^{n-1}(s), t\big) + B\big(V^{n-2}(s), V(t)\big) = v\big(B(V^{n-2}(s), t)\big) = 0 ,
\end{gather}
because $V^{n-2}(s) \in H^0(U, E_{n-1})$ by Definition \ref{def1}(2), and $E^\perp_1 = E_{n-1}$ (see Definition~\ref{parity}(2)). Similarly, we have
\begin{gather}\label{b2}
B\big(V^{n-2}(s), V(t)\big)+ B\big(V^{n-3}(s), V^2(t)\big) = v\big(B\big(V^{n-3}(s), V(t)\big)\big) = 0 ,
\end{gather}
because
\begin{itemize}\itemsep=0pt
\item $V^{n-3}(s) \in H^0(U, E_{n-2})$ (by Definition~\ref{def1}(1)),
\item $V(t) \in H^0(U, E_{2})$ (by Definition \ref{def1}(1)), and
\item $E^\perp_2 = E_{n-2}$ (Definition \ref{parity}(2)).
\end{itemize}
Iterating the argument for \eqref{b1} and \eqref{b2}, more generally we get that
\begin{gather}\label{b3}
{\mathcal T}_i := B\big(V^{n-i}(s), V^{i-1}(t)\big)+ B\big(V^{n-i-1}(s), V^i(t)\big) = 0
\end{gather}
for all $n-1 \leq i \geq 1$. Taking alternating sum, from \eqref{b3} we have
\[
(-1)^{i-1}{\mathcal T}_i = B\big(V^{n-1}(s), t\big) + (-1)^{n-2} B\big(s, V^{n-1}(t)\big) = 0 .
\]
This completes the proof.
\end{proof}

\begin{Proposition}\label{prop1}\mbox{}The section ${\mathbb S}'$ in \eqref{e7} lies in the subspace
\[
H^0\big(\Sigma, \operatorname{Sym}^2(E^*_1)\otimes K^{\otimes(n-1)}_\Sigma\big) \subset
H^0\big(\Sigma, (E^*_1)^{\otimes 2}\otimes K^{\otimes(n-1)}_\Sigma\big) .\]

This $K^{\otimes(n-1)}_\Sigma$-valued symmetric form ${\mathbb S}'$ on $E_1$ is fiberwise non-degenerate.
\end{Proposition}

\begin{proof}We shall first assume that $B$ is symmetric, hence an orthogonal form. In this case, from
Definition~\ref{parity}(1) we know that the integer $n$ in Definition~\ref{def1} is odd.
Since~$B$ is symmetric, and $n$ is odd, from \eqref{e8} in Lemma~\ref{lem22} it follows immediately that~${\mathbb S}'$ in~\eqref{e7}
is a section of \begin{gather*}\operatorname{Sym}^2(E^*_1)\otimes K^{\otimes(n-1)}_\Sigma \subset (E^*_1)^{\otimes 2}
\otimes K^{\otimes(n-1)}_\Sigma .\end{gather*}
Moreover, the $K^{\otimes(n-1)}_\Sigma$-valued bilinear form~${\mathbb S}'$ on $E_1$ is fiberwise
non-degenerate, because the homomorphism ${\mathbb S}'$ in~\eqref{e6} is an isomorphism.

Now assume that $B$ is a symplectic form. In this case, from Definition~\ref{parity}(1) we
know that the integer $n$ in Definition \ref{def1} is even.
Since $B$ is anti-symmetric, and $n$ is even, from \eqref{e8} it follows
that ${\mathbb S}'$ in \eqref{e7} is a section of \begin{gather*}\operatorname{Sym}^2(E^*_1)\otimes
K^{\otimes(n-1)}_\Sigma \subset (E^*_1)^{\otimes 2}\otimes K^{\otimes(n-1)}_\Sigma .\end{gather*}
Moreover, as in the previous case, the $K^{\otimes(n-1)}_\Sigma$-valued bilinear form ${\mathbb
S}'$ on $E_1$ is fiberwise non-degenerate, because the homomorphism ${\mathbb S}'$ in
\eqref{e6} is an isomorphism. \end{proof}

\section[Generalized $B$-opers and jet bundles]{Generalized $\boldsymbol{B}$-opers and jet bundles} \label{jets}

Generalized $B$-opers as introduced in Definition~\ref{def1} are closely related to {\it jet
bundles}, natural objects which have had different applications in geometry. In what follows,
we shall first recall the basic definitions and properties of jet bundles, and then study their
relation to generalized $B$-opers. The reader should refer, for instance, to~\cite{Bi3} for a~gentle introduction to the main ideas surrounding jet bundles which will come useful in the
present paper.

\subsection{Jet bundles}

We briefly recall the definition of a jet bundle of a holomorphic vector bundle on the Riemann surface $\Sigma$.
Consider the product $\Sigma\times\Sigma$, and let
\[p_i \colon \ \Sigma\times \Sigma \longrightarrow \Sigma , \qquad i = 1, 2\]
be the projection to the $i$-th factor. Let
\begin{gather}\label{delta}
\Delta := \{(x, x) \in \Sigma \,|\, x \in \Sigma\} \subset \Sigma\times \Sigma
\end{gather}
be the (reduced) diagonal divisor.

\begin{Definition}
Given a holomorphic vector bundle $W$ of rank $r_0$ on $\Sigma$ and any nonnegative
integer $m\geq 0$, the $m$th order {\em jet bundle} $J^m(W)$ of $W$ is the holomorphic
vector bundle of rank $(m + 1)r_0$ given by the direct image
\[
J^m(W) := p_{1*} \left(\frac{p^*_2W}{p^*_2W\otimes
{\mathcal O}_{\Sigma\times \Sigma}(-(m+1)\Delta)}\right)
 \longrightarrow \Sigma .
\]
\end{Definition}

The natural inclusion of ${\mathcal O}_{\Sigma\times \Sigma}(-(j+1)\Delta)$
in ${\mathcal O}_{\Sigma\times \Sigma}(-j\Delta)$ produces a surjective homomorphism
$J^{j}(W) \longrightarrow J^{j-1}(W)$.

\begin{Lemma}\label{kernel1}
For $j \geq 1$, the kernel of the above natural projection
\[J^{j}(W) \longrightarrow J^{j-1}(W)\]
 is the vector bundle $K^{\otimes j}_\Sigma\otimes W$.
\end{Lemma}

\begin{proof}
Whilst this is something well known, we shall include here a short proof in order to
be somewhat self-contained. We
will show this algebraically, following the proof given by Eisenbud and Harris for bundles of
principal parts \cite[p.~247]{MR3617981}. Since $\Sigma$ can be viewed as a smooth
projective curve, it suffices to prove the claim on $\mbox{Spec}(R)$, where
$R = \C[z]$. The vector bundle~$W$ on $R = \Sigma$ corresponds to a $R$-module,
which we also denote by $W$. The pullback bundle $p^*_2W = S \otimes_R W = R \otimes_\C W$ is
just an $S$-module from changing base ring. Let~$I$ be the ideal sheaf of the diagonal
$\Delta$. Then~$I$ is generated by $z \otimes 1 - 1 \otimes z$, where $z$ is the generator
of $R$. In this set-up, we see the $R$-module associated to the jet bundle is
\begin{align*}
J^m(W) = {p_1}_* \left( \frac{R \otimes_\C W}{R \otimes_\C W \otimes_\C I^{m+1}} \right)
 = W/I^{m+1}W
\end{align*}
viewed as an $R$-module with multiplication given by $r \longmapsto r \otimes_\C 1$.
Then the kernel of the natural projection $J^j(W) \longrightarrow J^{j-1}(W)$ is
given by $\big(I^m/I^{m+1}\big) \otimes_\C W$.
It is easy to see that $I^j/I^{j+1}$ is generated by $z^m \otimes 1 - 1 \otimes z^m$ and
is isomorphic to $\mbox{Sym}^m(\Omega_{R/\C}) = \Omega_{R/\C}^{\otimes m}$ as $R$-modules, where
the latter is generated by ${\rm d}z^{\otimes m}$. Finally, the canonical bundle coincides with the
sheaf of differentials for a curve. Hence the short exact sequence
\begin{gather*}
0 \longrightarrow K_\Sigma^{\otimes m} \otimes W \longrightarrow J^m(W)
 \longrightarrow J^{m-1}(W) \longrightarrow 0
\end{gather*}
is obtained.
\end{proof}

The following two definitions will be of much use here, and the reader should refer for
instance to \cite{Bi3} and references therein for further details on them.

\begin{Definition}Given holomorphic vector bundles $W$, $W'$ on $\Sigma$, we let
$\operatorname{Diff}^m_\Sigma(W, W') $ be the holomorphic vector bundle on~$\Sigma$ given by the
{\em sheaf of holomorphic differential operators} of order~$m$ from~$W$ to~$W'$, that is,
\begin{gather}\label{do}
\operatorname{Diff}^m_\Sigma(W, W') := \operatorname{Hom}\big(J^m(W), W'\big) \longrightarrow \Sigma .
\end{gather}
\end{Definition}

\begin{Definition}\label{symbol}The {\em symbol map} $\sigma\colon \operatorname{Diff}^m_\Sigma(W, W') \longrightarrow (T\Sigma)^{\otimes m}\otimes \operatorname{Hom}(W, W')$
is the composition of homomorphisms
\begin{gather*}
\operatorname{Diff}^m_\Sigma(W, W') = J^m(W)^*\otimes W' \longrightarrow
\big(K_\Sigma^{\otimes m}\otimes W\big)^*\otimes W' =(T\Sigma)^{\otimes m}\otimes \operatorname{Hom}(W, W'),
\end{gather*}
where the homomorphism is $J^m(W)^*\otimes W' \longrightarrow
\big(K_\Sigma^{\otimes m}\otimes W\big)^*\otimes W'$ is obtained from the inclusion
$K_\Sigma^{\otimes m}\otimes W \hookrightarrow J^m(W)$ in Lemma~\ref{kernel1}.
\end{Definition}

In what follows we shall construct an ${\mathcal O}_\Sigma$-linear homomorphism from the
holomorphic vector bundle $E$ of a generalized $B$-oper $(E, {\mathcal F}, D)$ to the jet
bundle $J^k(Q)$, where $k$ is any nonnegative integer and $Q = E/E_{n-1}$; this
will be done using the connection $D$.

Take a generalized $B$-oper $(E, {\mathcal F}, D)$.
Consider a point $x \in \Sigma$, and take a point in the corresponding fiber $v \in E_x$. We
shall denote by $\widetilde v$ the unique locally defined flat section of $E$ for the flat
connection $D$, for which ${\widetilde v}(x) = v$ (note that $\widetilde v$ is defined on any simply connected
open neighborhood of $x$). Consider the holomorphic section
\begin{gather*}q(\widetilde v) \colon \ \Sigma \longrightarrow Q := E/E_{n-1}\end{gather*}
of $Q$ defined around $x$, where $q\colon E \longrightarrow E/E_{n-1} = Q$ is the natural
quotient map (see \eqref{e4}). Now restricting $q(\widetilde v)$ to the $k$-th order
infinitesimal neighborhood of $x$ we get an element $q(\widetilde v)_k \in J^k(Q)_x$, and we
shall denote by
\begin{align}
f_k\colon \ E &\longrightarrow J^k(Q) , \nonumber \\
v &\longmapsto q(\widetilde v)_k,\label{e10}
\end{align}
the ${\mathcal O}_\Sigma$-linear homomorphism constructed this way.

From the construction of the homomorphisms $f_k$ in \eqref{e10} it follows immediately that the
following diagram of homomorphisms is commutative:
\[
\begin{matrix}
E& \stackrel{f_k}{\longrightarrow} & J^k(Q)\\
\big\Vert && \Big\downarrow\\
E& \stackrel{f_{k-j}}{\longrightarrow} & J^{k-j}(Q)
\end{matrix}
\]
for all $0 < j \leq k$, where the projection $J^k(Q) \longrightarrow
J^{k-j}(Q)$ is the composition of iterations of the projection in Lemma \ref{kernel1}.

The holomorphic vector
bundle $J^k(Q)$ has an increasing filtration of holomorphic subbundles of length $k+2$
\begin{gather}\label{e11}
0 = {\mathcal Q}_0 \subset {\mathcal Q}_1 \subset \cdots \subset
{\mathcal Q}_k \subset {\mathcal Q}_{k+1} := J^k(Q) ,
\end{gather}
where ${\mathcal Q}_i$, for every $0 \leq i \leq k$, is
the kernel of the natural projection $J^k(Q) \longrightarrow J^{k-i}(Q)$
obtained by iterating the projection in Lemma \ref{kernel1}.

The following result from \cite{Bi3} will be of much use.

\begin{Theorem}[{\cite[Theorem~4.2]{Bi3}}]\label{thm1}
The homomorphism $f_{n-1}$ in \eqref{e10} is an isomorphism. This isomorphism
$f_{n-1}$ takes the filtration $\mathcal F$ in \eqref{e4} to the filtration of
$J^{n-1}(Q)$ constructed in \eqref{e11}.
\end{Theorem}

\begin{Remark}
It should be clarified that the isomorphism $f_{n-1}$ in \eqref{e10} depends on the
connection $D$. More precisely, given $(E, {\mathcal F})$, if $D'$ and $D''$ are
two $B$-connections such that $(E, {\mathcal F}, D')$ and $(E, {\mathcal F}, D'')$
are generalized $B$-opers, then the two isomorphisms
\[
f'_{n-1} \colon \ E \longrightarrow J^k(Q) \qquad \text{and}\qquad f''_{n-1} \colon \ E \longrightarrow J^k(Q)
\]
corresponding to $D'$ and $D''$ respectively do not in general coincide.
\end{Remark}

The construction of the homomorphism $f_k$ in \eqref{e10} gives the following general result.

\begin{Proposition}\label{prop-j}Let $V$ and $W$ be holomorphic vector bundles on $\Sigma$, and let $D_W$ be a holomorphic connection on~$W$. Then for any integer $n \geq 1$, there is a natural holomorphic isomorphism
\begin{gather*}
\varphi \colon \ J^n(V)\otimes W \longrightarrow J^n(V\otimes W) .
\end{gather*}
\end{Proposition}

\begin{proof}
Take any point $x \in \Sigma$ and a vector $w \in W_x$ in the fiber of $W$ over $x$.
Let $s_w$ denote the unique holomorphic section of $W$, defined on a simply connected open neighborhood
of $x \in \Sigma$, satisfying the following two conditions:
\begin{itemize}\itemsep=0pt
\item $s_w(x) = w$, and

\item $s_w$ is flat with respect to the integrable connection $D_W$ on $W$.
\end{itemize}
Let $\widetilde{s}_w \in J^n(W)_x$ be the element obtained by restricting $s_w$ to the $n$-th
order infinitesimal neighborhood of $x$. Consequently, we get a holomorphic homomorphism
\begin{gather}\label{ej1}
\eta \colon \ W \longrightarrow J^n(W)
\end{gather}
that for any $x \in \Sigma$, sends any $w \in W_x$ to $\widetilde{s}_w \in J^n(W)_x$ constructed as above.

Given any two holomorphic vector bundles $V_1$ and $V_2$, there is a natural holomorphic homomorphism
\begin{gather*}
J^n(V_1)\otimes J^n(V_2) \longrightarrow J^n(V_1\otimes V_2) .
\end{gather*}
Now consider the composition of homomorphisms
\begin{gather*}
J^n(V)\otimes W \stackrel{{\rm Id}\otimes \eta}{\longrightarrow} J^n(V)\otimes J^n(W)
\longrightarrow J^n(V\otimes W) ,
\end{gather*}
where $\eta$ is the homomorphism in \eqref{ej1}. It is straightforward to check that this composition
of homomorphisms is fiberwise injective. This implies that it is an isomorphism, because
$\operatorname{rank}(J^n(V)\otimes W) = \operatorname{rank}(J^n(V\otimes W))$.
\end{proof}

The following is a special case of Proposition \ref{prop-j}.

\begin{Corollary}\label{cor-j}
Let $V$ be a holomorphic vector bundle on $\Sigma$, and ${\mathcal L}$ be a holomorphic
line bundle on $\Sigma$ of order $r$. Then there is a canonical holomorphic isomorphism
\begin{gather*}
J^n(V)\otimes{\mathcal L} \stackrel{\sim}{\longrightarrow} J^n(V\otimes{\mathcal L})
\end{gather*}
for every $j \geq 1$.
\end{Corollary}

\begin{proof}Take a holomorphic isomorphism
\begin{gather*}
\xi \colon \ {\mathcal O}_\Sigma \longrightarrow {\mathcal L}^{\otimes r} .
\end{gather*}
There is a unique holomorphic connection $D_{\mathcal L}$ on ${\mathcal L}$ such that the isomorphism
$\xi$ takes the trivial connection on ${\mathcal O}_\Sigma$, given by the de Rham differential~${\rm d}$,
to the connection on ${\mathcal L}^{\otimes r}$ induced by~$D_{\mathcal L}$. This connection~$D_{\mathcal L}$ does not depend on the choice of the isomorphism $\xi$, because any two choices of $\xi$ differ by an automorphism of~${\mathcal O}_\Sigma$ given by a nonzero scalar multiplication. Note that a nonzero scalar multiplication preserves the trivial connection on~${\mathcal O}_\Sigma$.

Since ${\mathcal L}$ is equipped with a canonical connection~$D_{\mathcal L}$, the required isomorphism
\begin{gather*}
J^n(V)\otimes{\mathcal L} \stackrel{\sim}{\longrightarrow} J^n(V\otimes{\mathcal L})
\end{gather*}
is given by Proposition~\ref{prop-j}.
\end{proof}

\subsection{A natural symmetric form}\label{spin}

Let $g$ be the genus of $\Sigma$. Recall that a theta characteristic on $\Sigma$ (or spin structure) is a holomorphic line bundle $\xi$ on $\Sigma$ of degree $g-1$ equipped with a holomorphic isomorphism of~$\xi^{\otimes 2}$ with~$K_\Sigma$. Fix a theta characteristic on $\Sigma$, and denote it by $K^{1/2}_\Sigma$. For any integer~$m$, the holomorphic line bundle $\big(K^{1/2}_\Sigma\big)^{\otimes m}$ will be denoted by~$K^{m/2}_\Sigma$.

Take a generalized $B$-oper $(E, {\mathcal F}, D)$ on $\Sigma$. For ease of notation we define the holomorphic vector bundle
\begin{gather}\label{CQ}
\CQ := Q\otimes K_\Sigma^{(n-1)/2} = (E/E_{n-1})\otimes K_\Sigma^{(n-1)/2} .
\end{gather}
Having studied the natural $K^{\otimes(n-1)}_\Sigma$-valued symmetric form
${\mathbb S}'$ on $E_1$ in Section \ref{sym11}, in what follows we shall construct an equivalent symmetric form on $\CQ$.

Using the isomorphism $E^*_1 = Q$
from Lemma \ref{lem11}, the section ${\mathbb S}'$ in \eqref{e7} produces a section
\begin{gather}\label{e13}
{\mathbb S} \in H^0\big(\Sigma, Q^{\otimes 2}\otimes K^{\otimes(n-1)}_\Sigma\big) = H^0\big(\Sigma, \CQ^{\otimes 2}\big) .
\end{gather}
In view of Proposition \ref{prop1}, the section $\mathbb S$ satisfies the following:

\begin{Proposition}\label{cor1}The section ${\mathbb S}$ in \eqref{e13} lies in the subspace
\[
H^0\big(\Sigma, \operatorname{Sym}^2(\CQ)\big) \subset
H^0\big(\Sigma, \CQ^{\otimes 2}\big) .
\]
Moreover, the symmetric bilinear form on $\CQ^*$ defined by $\mathbb S$ is fiberwise non-degenerate.
\end{Proposition}

\begin{proof}
This follows immediately from Proposition \ref{prop1} and Lemma \ref{lem11}.
\end{proof}

Since the symmetric bilinear form $\mathbb S$ in \eqref{e13} is fiberwise non-degenerate, it produces
a symmetric bilinear form
\[
{\mathbb S}^\vee \in H^0\big(\Sigma, (\CQ^*)^{\otimes 2}\big)
\]
on $\CQ$. In view of \eqref{CQ}, this ${\mathbb S}^\vee$ defines a homomorphism
\begin{gather}\label{e13b}
{\mathbb S}^\vee_0 \colon \ Q\otimes Q\otimes K^{\otimes (n-1)}_\Sigma \longrightarrow {\mathcal O}_\Sigma .
\end{gather}

\begin{Proposition}\label{lem1} A generalized $B$-oper $(E, {\mathcal F}, D)$ produces a holomorphic connection on the holomorphic vector bundle~$\CQ$. For the holomorphic connection on~$\operatorname{Sym}^2(\CQ)$ induced by this holomorphic connection on~$\CQ$, the section ${\mathbb S}$ in Proposition~{\rm \ref{cor1}} is flat.
\end{Proposition}

\begin{proof}Let $(E, {\mathcal F}, D)$ be a generalized $B$-oper on $\Sigma$. The triple
$(E, {\mathcal F}, D)$ produces a holomorphic differential operator ${\mathcal D}$ on
$\Sigma$ of order~$n$
\begin{gather}\label{dod}
{\mathcal D} \in H^0\big(\Sigma, \operatorname{Diff}^n_\Sigma\big(Q, Q\otimes K^n_\Sigma\big)\big)
\end{gather}
(see \cite[p.~18, equation~(4.6)]{Bi3}). As before, $\Delta \subset \Sigma\times\Sigma$ is the
divisor in~\eqref{delta}. Using $\mathcal D$ one can construct a holomorphic section of
$p^*_1\CQ\otimes p^*_2\CQ^*$ over the non-reduced divisor $2\Delta$ as in \cite[p.~27]{Bi3}
(note that in \cite{Bi3} this section is called $\kappa\vert_{2\Delta}\otimes s$). Moreover, as
shown in \cite[p.~27]{Bi3}, the restriction of this section to $\Delta \subset 2\Delta$
coincides with the identity map of $\CQ$. Therefore, this section of $p^*_1\CQ\otimes
p^*_2\CQ^*$ over $2\Delta$ defines a holomorphic connection on the holomorphic vector bundle
$\CQ$ \cite[p.~7]{Bi3}, \cite{De}; this connection on~$\CQ$ will be denoted by~$\nabla$.

Let $\widehat{\nabla}$ denote the holomorphic connection on $\operatorname{Sym}^2 (\CQ)$ induced by the
connection $\nabla$ on $\CQ$. To complete the proof of the proposition, we need to show that the section ${\mathbb S}$ of $\operatorname{Sym}^2(\CQ)$ in Proposition~\ref{cor1} is flat (covariant constant) with respect to this induced connection~$\widehat{\nabla}$. For that we need to recall the construction, as well as some properties, of the differential operator~$\mathcal D$ in~\eqref{dod}.

As before, $p_1$ and $p_2$ are the projections of $\Sigma\times\Sigma$ to the first and second factor respectively. The holomorphic differential operator $\mathcal D$ is given by a holomorphic section
\begin{gather*}
\kappa \in H^0\big((n+1)\Delta, p^*_1\big(K^{\otimes n}_\Sigma\otimes Q\big)\otimes
p^*_2(K_\Sigma\otimes Q^*) \otimes {\mathcal O}_{\Sigma\times\Sigma}((n+1)\Delta)\big)
\end{gather*}
over the nonreduced divisor $(n+1)\Delta$ \cite[p.~25, equation~(5.1)]{Bi3}. From \eqref{CQ} it follows immediately that
\begin{gather*}
p^*_1\big(K^{\otimes n}_\Sigma\otimes Q\big)\otimes
p^*_2(K_\Sigma\otimes Q^*) = p^*_1\big(K^{(n+1)/2}_\Sigma\otimes \CQ\big)\otimes
p^*_2\big(K^{(n+1)/2}_\Sigma\otimes \CQ^*\big) ,
\end{gather*}
and we conclude that
\begin{gather}\label{ka0}
\kappa \in H^0\big((n+1)\Delta, p^*_1\big(K^{(n+1)/2}_\Sigma\otimes \CQ\big)\otimes p^*_2
\big(K^{(n+1)/2}_\Sigma\otimes {\CQ}^*\big) \otimes {\mathcal O}_{\Sigma\times\Sigma}((n+1)\Delta)\big) .
\end{gather}

We note that the holomorphic line bundle $ p^*_1\big(K^{(n+1)/2}_\Sigma\big) \otimes \big(p^*_2 K^{(n+1)/2}_\Sigma\big)\otimes {\mathcal O}_{\Sigma\times\Sigma}((n+1)\Delta)$ has a canonical trivialization over $2\Delta$ \cite[p.~688, Theorem~2]{BR}. Since
\begin{gather*}
p^*_1\big(K^{(n+1)/2}_\Sigma\otimes \CQ\big)\otimes p^*_2
\big(K^{(n+1)/2}_\Sigma\otimes {\CQ}^*\big) \otimes {\mathcal O}_{\Sigma\times\Sigma}((n+1)\Delta)\\
\qquad{}
= (p^*_1\CQ)\otimes (p^*_2\CQ^*)\otimes p^*_1\big(K^{(n+1)/2}_\Sigma\big) \otimes \big(p^*_2 K^{(n+1)/2}_\Sigma\big)\otimes
{\mathcal O}_{\Sigma\times\Sigma}((n+1)\Delta) ,
\end{gather*}
using this trivialization of $ p^*_1\big(K^{(n+1)/2}_\Sigma\big) \otimes \big(p^*_2 K^{(n+1)/2}_\Sigma\big)\otimes
{\mathcal O}_{\Sigma\times\Sigma}((n+1)\Delta)$ over $2\Delta$, the section~$\kappa$ in~\eqref{ka0} produces a section of $(p^*_1\CQ)\otimes (p^*_2\CQ^*)$. This section of $(p^*_1\CQ)\otimes (p^*_2\CQ^*)$ over $2\Delta$ gives the holomorphic connection on~$\nabla$ on~$\CQ$.

On the other hand, $\CQ$ is identified with its dual~${\CQ}^*$ using the pairing~$\mathbb S$ in Proposition~\ref{cor1}. Consequently, we conclude that
\begin{gather}\label{ka}
\kappa \in H^0\big((n+1)\Delta, p^*_1\big(K^{(n+1)/2}_\Sigma\otimes \CQ\big)\otimes p^*_2
\big(K^{(n+1)/2}_\Sigma\otimes {\CQ}\big)\otimes {\mathcal O}_{\Sigma\times\Sigma}((n+1)\Delta)\big) .
\end{gather}
Now from the construction of $\kappa$ it follows that the section $\kappa$ is symmetric, meaning
\begin{gather}\label{eta}
\eta^*\kappa = \kappa ,
\end{gather}
where $\eta \colon \Sigma\times\Sigma \longrightarrow \Sigma\times\Sigma$ is the involution defined by $(x, y) \longmapsto (y, x)$; in~\eqref{eta}, the section $\kappa$ is considered as the section in~\eqref{ka}.

In view of $\eta$, the following lemma implies that the section~${\mathbb S}$ of $\operatorname{Sym}^2 (\CQ)$ is covariantly constant with respect to the connection $\widehat{\nabla}$.
\end{proof}

\begin{Lemma}\looseness=-1 Let $F$ be a holomorphic vector bundle over $\Sigma$ equipped with a holomorphic connec\-tion~$\nabla_F$. Let $\theta^F$ be the section of $(p^*_1 F)\otimes (p^*_2 F^*)$ over $2\Delta$ corresponding to the connection~$\nabla_F$.~Let
\begin{gather*}
B \in H^0\big(\Sigma, \operatorname{Sym}^2(F)\big)
\end{gather*}
be a nondegenerate symmetric bilinear form on~$F^*$. Then the following statements hold:
\begin{itemize}\itemsep=0pt
\item Using $B$, the section $\theta^F$ produces a section $\widetilde{\theta}^F$ of
$(p^*_1 F)\otimes (p^*_2 F)$ over $2\Delta$.

\item For the connection $\widetilde{\nabla}_F$ on $F\otimes F$ induced by $\nabla_F$,
\begin{gather*}
\widetilde{\nabla}_F (B) = \widetilde{\theta}^F- \eta^*\widetilde{\theta}^F
 \in H^0\big(\Sigma, \operatorname{Sym}^2(F)\otimes K_\Sigma\big) ,
\end{gather*}
where $\eta$ is the involution of $\Sigma\times\Sigma$ in \eqref{eta}.
\end{itemize}
In particular, $B$ is covariant constant with respect to the connection $\widetilde{\nabla}_F$ if and
only if $\eta^*\widetilde{\theta}^F = \widetilde{\theta}^F$.
\end{Lemma}

\begin{proof}Since $B$ identifies $F$ with $F^*$, we have
$(p^*_1 F)\otimes (p^*_2 F^*) = (p^*_1 F)\otimes (p^*_2 F)$. So
$\theta^F$ produces a section $\widetilde{\theta}^F$ of
$(p^*_1 F)\otimes (p^*_2 F)$ over $2\Delta$.

The restriction of $\theta^F$ to $\Delta \subset 2\Delta$ is $\text{Id}_F$. Hence we have
\begin{gather*}\widetilde{\theta}^F\vert_\Delta = B = \big(\eta^*\widetilde{\theta}^F\big)\vert_\Delta .\end{gather*}
This, and the fact that ${\mathcal O}_{\Sigma\times\Sigma}(-\Delta)\vert_\Delta = K_\Delta$ with $K_\Delta$
being the holomorphic cotangent bundle of $\Delta$, together imply that
\begin{gather*}
\widetilde{\theta}^F- \eta^*\widetilde{\theta}^F
 \in H^0\big(\Sigma, \operatorname{Sym}^2(F)\otimes K_\Sigma\big)
\end{gather*}
after we identify $\Delta$ with $\Sigma$ using the natural map $x \longmapsto (x, x)$. Now it is
straightforward to check that $\widetilde{\theta}^F- \eta^*\widetilde{\theta}^F = \widetilde{\nabla}_F (B)$.
\end{proof}

Let
\begin{gather*}
{\mathbb H} \colon \ Q\otimes \big(Q\otimes K^n_\Sigma\big) \longrightarrow K_\Sigma
\end{gather*}
be the homomorphism defined by
\begin{gather*}
Q\otimes \big(Q\otimes K^n_\Sigma\big) = \big(Q\otimes \big(Q\otimes K^{\otimes (n-1)}_\Sigma\big)\big)\otimes
K_\Sigma \stackrel{{\mathbb S}^\vee_0\otimes {\rm Id}}{\longrightarrow} K_\Sigma ,
\end{gather*}
where ${\mathbb S}^\vee_0$ is the homomorphism in~\eqref{e13b}. Consider $\mathbb H$ as a
paring $\langle-, - \rangle$ between $Q$ and $Q\otimes K^n_\Sigma$ with values in~$K_\Sigma$.
Then it can be shown that the differential operator ${\mathcal D}$ in~\eqref{dod} satisfies
the equation
\begin{gather}\label{dei}
\langle {\mathcal D}(s), t\rangle = \langle s, {\mathcal D}(t)\rangle
\end{gather}
for any locally defined holomorphic sections $s$ and $t$ of $Q$. Indeed, \eqref{dei} is an
exact reformulation of~\eqref{eta}. In other words, the differential operator ${\mathcal D}$ is ``self-adjoint''.

\begin{Remark}Given the above defined pairing $\langle- , -\rangle$
between $Q$ and $Q\otimes K^n_\Sigma$ with values in~$K_\Sigma$, for any differential operator
\[
D_n \in H^0\big(\Sigma, \operatorname{Diff}^n_\Sigma\big(Q, Q\otimes K^n_\Sigma\big)\big) ,
\]
the adjoint differential operator $D^*_n$ defined by the equation
\[
\langle D_n(s), t\rangle = \langle s, D^*_n(t)\rangle ,
\]
where $s$ and $t$ are any locally defined holomorphic sections of~$Q$, is similar to the classical Lagrange adjoint~\cite{r8}. More precisely, when the rank of~$Q$ is one, the above defined adjoint map coincides with the Lagrange adjoint.
\end{Remark}

\subsection{Symplectic and orthogonal opers}\label{unique}

In what follows, using the connection $\nabla$ on $\CQ$ from Proposition~\ref{lem1}, we shall
construct a~holomorphic isomorphism between the bundles $J^{n-1}(Q)$ and $\CQ\otimes
J^{n-1}\big(K^{(1-n)/2}_\Sigma\big)$, which in turn would enable us to express generalized $B$-opers in terms of classical opers together with a fiberwise non-degenerate symmetric bilinear form on~$ \CQ ^*$,
and a holomorphic connection on~$\CQ^*$ that preserves this form.

As before, $(E, {\mathcal F}, D)$ is a generalized $B$-oper on $\Sigma$.

Given any point $x \in \Sigma$ and $v \in \CQ_x$, let $\widetilde{v}$ be the unique flat section of~$\CQ$, defined around~$x$, such that $\widetilde{v}(x) = v$. Then, for each~$k$, just as the map $f_k$ in~\eqref{e10} is constructed, we have a~homomorphism
\begin{gather*}
\gamma^k_\nabla \colon \ \CQ \longrightarrow J^{k}(\CQ)
\end{gather*}
that sends any $v \in \CQ_x$, $x \in \Sigma$, to the element of~$J^{k}(\CQ)_x$ obtained by restricting the flat section~$\widetilde{v}$ to the $k$-th order infinitesimal neighborhood of~$x$. Moreover, we shall denote by
\begin{gather}\label{gn}
\Gamma^k_\nabla \colon \ \CQ\otimes J^{k}\big(K^{(1-n)/2}_\Sigma\big) \longrightarrow J^{k}(Q)
\end{gather}
the homomorphism given by the composition of homomorphisms
\begin{gather*}
\CQ\otimes J^{k}\big(K^{(1-n)/2}_\Sigma\big) \stackrel{\gamma^k_\nabla
\otimes{\rm Id}}{\longrightarrow} J^{k}(\CQ)\otimes J^{k}\big(K^{(1-n)/2}_\Sigma\big)
\longrightarrow J^{k}\big(\CQ\otimes K^{(1-n)/2}_\Sigma\big) = J^{k}(Q) ,
\end{gather*}
where the homomorphism $J^{k}(\CQ)\otimes J^{k}\big(K^{(1-n)/2}_\Sigma\big)
\longrightarrow J^{k}\big(\CQ\otimes K^{(1-n)/2}_\Sigma\big)$ is the
natural homomorphism $J^k(A)\otimes J^k(B) \longrightarrow
J^k(A\otimes B)$ for any vector bundles $A$, $B$ on~$\Sigma$.
For ease of notation, in what follows we denote by~$\CJk$ the jet bundle
\begin{gather}\label{jk}
\CJk := J^{k-1}\big(K^{(1-k)/2}_\Sigma\big) .
\end{gather}
Then, the following lemma is established.

\begin{Lemma}\label{lem3}The homomorphism
\begin{gather*}
\Gamma^k_\nabla\colon \ \CQ\otimes J^{k}\big(K^{(1-n)/2}_\Sigma\big)
 \longrightarrow J^{k}(Q)
\end{gather*}
in \eqref{gn} is an isomorphism for each $k$. Hence the isomorphism $f_{n-1}$ in
Theorem~{\rm \ref{thm1}} produces a holomorphic isomorphism of $E = J^{n-1}(Q)$ with
$\CQ\otimes\CJn$, where $\CJn$ is defined in~\eqref{jk}.
\end{Lemma}

\begin{proof}We will show that the homomorphism~$\Gamma^k_\nabla$ is fiberwise injective. For that,
take any point $x \in \Sigma$, and consider the restriction of $\CQ$ to a simply
connected open neighborhood~$U_x$ of~$x$. Recall from Proposition~\ref{lem1} that $\CQ$ is
equipped with a holomorphic connection $\nabla$. As mentioned earlier, any holomorphic connection
on a Riemann surface is automatically flat. Using this flat connection $\nabla$ on $\CQ$, we trivialize
the restriction $\CQ\vert_{U_x}$ of~$\CQ$ to $U_x$ (this is possible because~$U_x$ is simply
connected). Since $\CQ \otimes K_\Sigma^{(1-n)/2} = Q$ (see~\eqref{CQ}), using this trivialization of $\CQ\vert_{U_x}$, the restriction $Q\vert_{U_x}$ of~$Q$
to $U_x$ gets identified with $\big(K_\Sigma^{(1-n)/2}\big)^{\oplus r}\vert_{U_x}$, where $r = \operatorname{rank}(Q)$. Using this identification of $Q\vert_{U_x}$ with $\big(K_\Sigma^{(1-n)/2}\big)^{\oplus r}\vert_{U_x}$, the restriction of the homomorphism $\Gamma^k_\nabla$ in~\eqref{gn} to $U_x$ gets identified with the
natural homomorphism
\begin{gather}\label{j1}
J^{k}\big(K^{(1-n)/2}_\Sigma\big)^{\oplus r} \longrightarrow J^{k}\big(\big(K_\Sigma^{(1-n)/2}\big)^{\oplus r}\big) .
\end{gather}
The homomorphism in~\eqref{j1} is evidently fiberwise injective. Consequently,
$\Gamma^k_\nabla$ is fiberwise injective. The homomorphism in~\eqref{j1} is also fiberwise
surjective, using which it can be deduced that~$\Gamma^k_\nabla$ is an isomorphism. Alternatively, since
\begin{gather*}\operatorname{rank}\big(\CQ\otimes J^{k}\big(K^{(1-n)/2}_\Sigma\big)\big) = (k+1)r =
\operatorname{rank}\big(J^{k}(Q)\big) ,\end{gather*}
$\Gamma^k_\nabla$ is fiberwise surjective, because it is fiberwise injective.
\end{proof}

\begin{Corollary}\label{corji} For a generalized $B$-oper $(E, {\mathcal F}, D)$ on $\Sigma$, the holomorphic vector bundle $E$ is canonically identified with $\CQ\otimes J^{n-1}\big(K^{(1-n)/2}_\Sigma\big)$.
\end{Corollary}

\begin{proof} This follows from the combination of Theorem~\ref{thm1} and Lemma~\ref{lem3}.
\end{proof}

\begin{Lemma}\label{lem111} Associated to $(E, {\mathcal F}, D)$, there is a natural homomorphism
\begin{gather*}
\tau \colon \ \operatorname{Diff}^n_\Sigma\big(Q, Q\otimes K^n_\Sigma\big) \longrightarrow
\operatorname{Diff}^n_\Sigma\big(K^{(1-n)/2}_\Sigma , K^{(n+1)/2}_\Sigma\big) .
\end{gather*}
\end{Lemma}

\begin{proof}Combining the isomorphism in Lemma~\ref{lem3} for $k \!=\! n$ with the
definition of $\operatorname{Diff}^m_\Sigma(W, W')$ in~\eqref{do} for $m=n$, $W=Q$ and $W'=Q\otimes K^n_\Sigma$, we have
\begin{align}
\operatorname{Diff}^n_\Sigma\big(Q, Q\otimes K^n_\Sigma\big) &= \big(Q\otimes K^n_\Sigma\big) \otimes J^n(Q)^*
= Q\otimes K^n_\Sigma \otimes \big(\CQ\otimes J^{n}\big(K^{(1-n)/2}_\Sigma\big)\big)^*\nonumber\\
&= Q\otimes K^n_\Sigma \otimes \big(Q\otimes K^{(n-1)/2}_\Sigma
\otimes J^{n}\big(K^{(1-n)/2}_\Sigma\big)\big)^*\nonumber\\
&= \operatorname{Diff}^n_\Sigma \big(K^{(1-n)/2}_\Sigma , K^{(n+1)/2}_\Sigma\big)\otimes\operatorname{End}(Q) . \label{ij}
\end{align}
Using the trace homomorphism
\begin{align*}
\operatorname{End}(Q) &\longrightarrow {\mathcal O}_\Sigma ,\nonumber\\
 s &\longmapsto
\frac{1}{\operatorname{rank}(Q)}\operatorname{trace}(s) ,\nonumber
\end{align*}
the isomorphism in~\eqref{ij} gives the homomorphism
\begin{gather*}
\tau \colon \ \operatorname{Diff}^n_\Sigma\big(Q, Q\otimes K^n_\Sigma\big) \longrightarrow
\operatorname{Diff}^n_\Sigma \big(K^{(1-n)/2}_\Sigma , K^{(n+1)/2}_\Sigma\big)
\end{gather*}
as required.
\end{proof}

Although the holomorphic connection $\nabla$ on $\CQ$ does not explicitly arise in the statement or proof
of Lemma \ref{lem111}, the homomorphism $\tau$ in Lemma \ref{lem111} crucially uses $\nabla$. Indeed,
the isomorphism $\Gamma^k_\nabla$ in Lemma \ref{lem3}, whose construction uses $\nabla$, is the
key ingredient in the construction of $\tau$ in Lemma \ref{lem111}.

Through the map $\tau$ from Lemma \ref{lem111} we shall see how generalized $B$-opers relate
to the classical ${\rm Sp}(n,\C)$ and ${\rm SO}(n,\C)$-opers as introduced in Section \ref{intro}.

\begin{Proposition}\label{prop2} Consider the differential operator $\mathcal D$ in~\eqref{dod}. For the map $\tau$ in Lemma~{\rm \ref{lem111}}, The differential operator
\begin{gather*}
\tau({\mathcal D}) \in H^0\big(\Sigma, \operatorname{Diff}^n_\Sigma\big(K^{(1-n)/2}_\Sigma,
K^{(n+1)/2}_\Sigma\big)\big)
\end{gather*}
is a classical ${\rm Sp}(n, {\mathbb C})$-oper for $n$ even, and a~classical ${\rm SO}(n, {\mathbb C})$-oper for $n$ odd.
\end{Proposition}

\begin{proof}Holomorphic differential operators on $\Sigma$ are given by holomorphic section of suitable bundles on the neighborhood of the diagonal (see \cite[p.~25, equation~(5.1)]{Bi3} for the precise statement). First we shall describe the homomorphism~$\tau$ in Lemma~\ref{lem111} in terms of such sections.

Recall from \eqref{ka} that the differential operator $\mathcal D$ in~\eqref{dod} is given by the section
\begin{gather*}
\kappa \in H^0\big((n+1)\Delta, p^*_1\big(K^{(n+1)/2}_\Sigma\otimes \CQ\big)\otimes p^*_2
\big(K^{(n+1)/2}_\Sigma\otimes {\CQ}\big)\otimes {\mathcal O}_{\Sigma\times\Sigma}((n+1)\Delta)\big)\\
\qquad {} =
H^0\big((n+1)\Delta, \big(p^*_1 K^{(n+1)/2}_\Sigma\big)\otimes \big(p^*_2
K^{(n+1)/2}_\Sigma\big)\otimes (p^*_1\CQ)\otimes(p^*_2 \CQ) \otimes {\mathcal O}_{\Sigma\times\Sigma}((n+1)\Delta)\big) .
\end{gather*}
It can be shown that using the flat connection $\nabla$ on $\CQ$, the two vector bundles
$p^*_1\CQ$ and $p^*_2\CQ$ are identified on some open neighborhood of the diagonal $\Delta
\subset \Sigma\times\Sigma$. Indeed, take an open subset $U_\Delta \subset
\Sigma\times\Sigma$ containing $\Delta$ such that $\Delta$ is a deformation retraction of $U_\Delta$.
For $i = 1, 2$, let $q_i\colon U_\Delta \longrightarrow \Sigma$ be the projection to the $i$-th
factor. Consider the two flat bundles $(q^*_1\CQ, q^*_1\nabla)$ and $(q^*_2\CQ, q^*_2\nabla)$
on~$U_\Delta$. On the diagonal $\Delta \subset U_\Delta$, both are evidently identified with the flat bundle
$(\CQ, \nabla)$ (once we identify~$\Delta$ with $\Sigma$ using the map $x \longmapsto (x, x)$). Since~$\Delta$ is a deformation retraction of~$U_\Delta$, this isomorphism between
$(q^*_1\CQ, q^*_1\nabla)\vert_\Delta$ and $(q^*_2\CQ, q^*_2\nabla)\vert_\Delta$ extends to an
isomorphism between $(q^*_1\CQ, q^*_1\nabla)$ and $(q^*_2\CQ, q^*_2\nabla)$ over the entire
$U_\Delta$. Using this isomorphism between $q^*_1\CQ$ and $q^*_2\CQ$, the above section $\kappa$ becomes
a section
\begin{gather*}
\kappa \in H^0\big((n+1)\Delta, \big(p^*_1 K^{(n+1)/2}_\Sigma\big)\otimes \big(p^*_2
K^{(n+1)/2}_\Sigma\big)\otimes p^*_1(\CQ\otimes\CQ) \otimes {\mathcal O}_{\Sigma\times\Sigma}((n+1)\Delta)\big) .
\end{gather*}
Now using the symmetric bilinear form ${\mathbb S}^*$ on $\CQ$ (see \eqref{e13}), the above section $\kappa$
produces a~section
\begin{gather*}
\widetilde{\kappa} \in H^0\big((n+1)\Delta, \big(p^*_1 K^{(n+1)/2}_\Sigma\big)\otimes \big(p^*_2
K^{(n+1)/2}_\Sigma\big)\otimes {\mathcal O}_{\Sigma\times\Sigma}((n+1)\Delta)\big) .
\end{gather*}
Using the isomorphism in \cite[p.~25, equation~(5.1)]{Bi3}, this section $\widetilde{\kappa}$ produces a
holomorphic differential operator
\begin{gather}\label{wtd}
\widetilde{\mathcal D} \in H^0\big(\Sigma, \operatorname{Diff}^n_\Sigma\big(K^{(1-n)/2}_\Sigma,
K^{(n+1)/2}_\Sigma\big)\big) .
\end{gather}
The differential operator $\widetilde{\mathcal D}$ in \eqref{wtd} coincides with
the differential operator $\tau({\mathcal D})$ in the statement of Proposition \ref{prop2}.

Let
\begin{gather*}
K^{(1-n)/2}_\Sigma\otimes K^{(n+1)/2}_\Sigma \longrightarrow K_\Sigma
\end{gather*}
be the natural homomorphism; we shall denote this $K_\Sigma$-valued pairing between
$K^{(1-n)/2}_\Sigma$ and $K^{(n+1)/2}_\Sigma$ by $\langle-, -\rangle$. Now from~\eqref{dei}
it follows that
\begin{gather}\label{dei2}
\langle \widetilde{\mathcal D}(s), t\rangle = \langle s, \widetilde{\mathcal D}(t)\rangle
\end{gather}
for locally defined holomorphic sections $s$ and $t$ of $K^{(1-n)/2}_\Sigma$.

The symbol of the differential operator $\mathcal D$ is $\text{Id}_Q$ \cite[p.~18]{Bi3} (see
the paragraph following~(4.6)). From this it follows that the symbol of $\widetilde{\mathcal
D}$ is the constant function $1$. Such a differential operator produces a holomorphic connection
on the jet bundle $J^{n-1}(K^{(1-n)/2}_\Sigma)$ (see \cite[p.~15, equation~(4.3)]{Bi2}; the
spaces~$\mathcal A$ and~$\mathcal B$ in \cite[equation~(4.3)]{Bi2} are defined in \cite[p.~13]{Bi2}). The holomorphic connection on~$J^{n-1}(K^{(1-n)/2}_\Sigma)$ given by the differential operator
$\widetilde{\mathcal D}$ will be denoted by~$\widehat D$.

Consequently, all the three vector bundles in Corollary \ref{corji}, namely $E$, $\CQ$ and
$J^{n{-}1}\!\big(K^{(1{-}n)/2}_\Sigma\big),\!$ are equipped with holomorphic connections. The
isomorphism in Corollary~\ref{corji} between $E$ and
$\CQ\otimes J^{n-1}\big(K^{(1-n)/2}_\Sigma\big)$ is connection preserving. In particular,
the above connection $\widehat D$ is an ${\rm SL}(n,{\mathbb C})$-connection; recall that~$D$ on~$E$ is an ${\rm SL}(nr,{\mathbb C})$-connection and $\nabla$ on $\CQ$ is an ${\rm SL}(r,{\mathbb C})$-connection, where $r = \operatorname{rank}(\CQ)$.

Therefore, $\widetilde{\mathcal D}$ defines an ${\rm SL}(n,{\mathbb C})$-oper on $\Sigma$.
We recall that the ${\rm Sp}(n, {\mathbb C})$ opers ($n$ is an even integer) and
${\rm SO}(n, {\mathbb C})$ opers ($n \geq 5$ is an odd integer) have the following property:
Let $\big(F, \nabla^F\big)$ be a holomorphic principal ${\rm Sp}(n, {\mathbb C})$-bundle, or a holomorphic
principal ${\rm SO}(n, {\mathbb C})$-bundle, $F$ equipped with a holomorphic connection $\nabla^F$, giving
such an oper. Consider the holomorphic principal ${\rm SL}(n, {\mathbb C})$-bundle equipped
with a holomorphic connection obtained by extending the structure group of $\big(F, \nabla^F\big)$ using the
natural inclusions of ${\rm Sp}(n, {\mathbb C})$ and ${\rm SO}(n, {\mathbb C})$ in
${\rm SL}(n, {\mathbb C})$. Then this holomorphic principal ${\rm SL}(n, {\mathbb C})$-bundle equipped
with a holomorphic connection is an ${\rm SL}(n, {\mathbb C})$ oper. It should be mentioned that this is not true for ${\rm SO}(2n, {\mathbb C})$-opers. This is one of the reasons why our method does not apply to the even orthogonal cases, which shall be treated separatedly in future work~\cite{BSY}.

Since $\widetilde{\mathcal D}$ satisfies \eqref{dei2}, the ${\rm SL}(n,{\mathbb C})$-oper
given by it is actually an ${\rm Sp}(n, {\mathbb C})$-oper when $n$ is even, and it is an
${\rm SO}(n, {\mathbb C})$-oper when $n$ is odd. In other words, the holomorphic
connection $\widehat D$ on $J^{n-1}\big(K^{(1-n)/2}_\Sigma\big)$ defines a
${\rm Sp}(n, {\mathbb C})$-oper when $n$ is even, and it defines a
${\rm SO}(n, {\mathbb C})$-oper when $n$ is odd. To construct the bilinear form on
$J^{n-1}\big(K^{(1-n)/2}_\Sigma\big)$ for this ${\rm Sp}(n, {\mathbb C})$ or
${\rm SO}(n, {\mathbb C})$ oper structure, consider the isomorphism
\begin{gather*}
\CQ\otimes J^{n-1}\big(K^{(1-n)/2}_\Sigma\big) \stackrel{\sim}{\longrightarrow} E
\end{gather*}
in Corollary~\ref{corji}. This produces an injective homomorphism
\begin{gather*}
J^{n-1}\big(K^{(1-n)/2}_\Sigma\big) \hookrightarrow E\otimes{\CQ}^* .
\end{gather*}
The bilinear form $B$ on $E$ and the bilinear form ${\mathbb S}$ on ${\CQ}^*$ (see
\eqref{e13} and Proposition \ref{cor1}) together produce a bilinear form on~$E\otimes{\CQ}^*$. The bilinear form on $J^{n-1}\big(K^{(1-n)/2}_\Sigma\big)$ is the restriction of this bilinear form
on $E\otimes{\CQ}^*$. (See also Remark~\ref{reml}.)

Let $\nabla$ denote the holomorphic connection on ${\CQ}^*$ induced
by the holomorphic connection $\nabla$ on $\CQ$.
Since the holomorphic connection $D$ (respectively, $\nabla^*$)
on $E$ (respectively, ${\CQ}^*$) preserves the bilinear form $B$ (respectively, $\mathbb S$)
on $E$ (respectively, ${\CQ}^*$), the holomorphic connection $\widehat D$ on
$J^{n-1}\big(K^{(1-n)/2}_\Sigma\big)$ preserves the bilinear form on $J^{n-1}\big(K^{(1-n)/2}_\Sigma\big)$
constructed above.
\end{proof}

From the above results, one has the following:

\begin{Theorem}\label{prop3}Let $(E, {\mathcal F}, D)$ be a generalized $B$-oper over $\Sigma$ of filtration length $n > 0$. Then, the following three objects are canonically associated to $(E, {\mathcal F}, D)$:
\begin{enumerate}\itemsep=0pt
\item[$1)$] a fiberwise non-degenerate symmetric bilinear form on
$ \CQ = E/E_{n-1}\otimes K^{\frac{n-1}{2}}$,
\item[$2)$] a holomorphic connection on $\CQ$ that preserves this
bilinear form, and
\item[$3)$] a classical ${\rm Sp}(n, {\mathbb C})$-oper for $n$ even and a classical
${\rm SO}(n, {\mathbb C})$-oper for $n$ odd.
\end{enumerate}
\end{Theorem}

\begin{proof}Consider the symmetric bilinear form $\mathbb S$ on ${\CQ}^*$ in~\eqref{e13}.
Since $\mathbb S$ is nondegenerate (Proposition~\ref{cor1}), it
produces a nondegenerate symmetric bilinear form on the dual vector bundle~${\CQ}$.

Then, from Proposition~\ref{lem1}, there
is a holomorphic connection $\nabla$ on
$\CQ$ which preserves the symmetric bilinear form on $\CQ$ given by $\mathbb S$.

Finally, the classical ${\rm Sp}(n, {\mathbb C})$-oper for $n$ even, and classical
${\rm SO}(n, {\mathbb C})$-oper for $n$ odd, are given by the differential operator
$\tau({\mathcal D})$ in the statement of Proposition~\ref{prop2}.
\end{proof}

The correspondence in Theorem~\ref{prop3} can be shown to be a one-to-one correspondence, and
therefore in the next section we shall prove a converse of Theorem~\ref{prop3}.

\section[Generalized $B$-opers and projective structures]{Generalized $\boldsymbol{B}$-opers and projective structures}\label{construction}

\subsection[A projective structure from a $B$-oper]{A projective structure from a $\boldsymbol{B}$-oper}

Given a projective structure $P$ on $\Sigma$, from \cite[p.~13, equation~(3.6)]{Bi2} we have a ($P$~dependent) decomposition of the space of differential operators
\begin{gather}\label{dodo}
H^0\big(\Sigma, \operatorname{Diff}^n_\Sigma\big(K^{(1-n)/2}_\Sigma , K^{(n+1)/2}_\Sigma\big)\big)
 = \bigoplus_{j=0}^n H^0\big(\Sigma, K^{\otimes j}_\Sigma\big).
\end{gather}
The component in $H^0(\Sigma, {\mathcal O}_\Sigma)$ corresponding to a differential operator
\[D' \in H^0\big(\Sigma, \operatorname{Diff}^n_\Sigma\big(K^{(1-n)/2}_\Sigma , K^{(n+1)/2}_\Sigma\big)\big)\]
is the symbol $\sigma(D')$ of $D'$.
Moreover, given a differential operator
\[
{\mathbf D} \in H^0\big(\Sigma, \operatorname{Diff}^n_\Sigma\big(K^{(1-n)/2}_\Sigma , K^{(n+1)/2}_\Sigma\big)\big),
\]
whose symbol is a nonzero (constant) function, there is a unique projective structure
$P_{\mathbf D}$ on~$\Sigma$ such that the component of $\mathbf D$ in $H^0\big(\Sigma,
K^2_\Sigma\big)$ vanishes identically for the decomposition of $H^0\big(\Sigma, {\rm
Diff}^n_\Sigma\big(K^{(1-n)/2}_\Sigma , K^{(n+1)/2}_\Sigma\big)\big)$ (as in~\eqref{dodo}) associated to~$P_{\mathbf D}$ (see \cite[p.~14, equation~(3.7)]{Bi2}).

\begin{Definition}\label{defP}
Consider the differential operator ${\mathcal D} \in H^0\big(\Sigma, \operatorname{Diff}^n_\Sigma\big(Q,
Q\otimes K^n_\Sigma\big)\big) $ in \eqref{dod}, and the map $\tau$ in Lemma~\ref{lem111}. We shall denote by~$\mathbf P$ the unique projective structure corresponding to the differential operator
\begin{gather*}
\tau({\mathcal D}) \in H^0\big(\Sigma, \operatorname{Diff}^n_\Sigma\big(K^{(1-n)/2}_\Sigma,
K^{(n+1)/2}_\Sigma\big)\big),
\end{gather*}
for which the component of $\tau({\mathcal D})$ in $H^0\big(\Sigma, K^2_\Sigma\big)$ vanishes
identically for the decomposition of $H^0\big(\Sigma, \operatorname{Diff}^n_\Sigma\big(K^{(1-n)/2}_\Sigma,
K^{(n+1)/2}_\Sigma\big)\big)$ as in~\eqref{dodo} associated to~$\mathbf P$.
\end{Definition}

\begin{Lemma}\label{lemB}The projective structure $\mathbf P$ induces a canonical bilinear form
\begin{gather*}
{\mathbf B}_n \colon \ \CJn\otimes\CJn \longrightarrow
{\mathcal O}_\Sigma
\end{gather*}
on $\CJn$ $($defined in~\eqref{jk}$)$ which is orthogonal when~$n$ is odd and symplectic when~$n$ is even.
\end{Lemma}

\begin{proof} Any two choices of theta characteristic on $\Sigma$ differ by tensoring with a holomorphic line bundle of order two on~$\Sigma$. If $K^{1/2}$ and ${\mathbb K}^{1/2} = K^{1/2}\otimes {\mathcal L}$
are two theta characteristics on~$\Sigma$, where ${\mathcal L}$ is a holomorphic line bundle
of order two, then ${\mathbb K}^{(1-k)/2}_\Sigma = K^{(1-k)/2}_\Sigma\otimes
({\mathcal L}^*)^{\otimes (k-1)}$. Therefore, from Corollary~\ref{cor-j} we conclude that
\begin{gather}\label{jkl}
J^{k-1}\big({\mathbb K}^{(1-k)/2}_\Sigma\big) =
J^{k-1}\big(K^{(1-k)/2}_\Sigma\big)\otimes({\mathcal L}^*)^{\otimes (k-1)} =
\CJk\otimes({\mathcal L}^*)^{\otimes (k-1)} ,
\end{gather}
where $\CJk$ is defined in~\eqref{jk}.

In view of \eqref{jkl}, using \cite[p.~10, Theorem~3.7]{Bi1}, the projective structure~$\mathbf P$ from Definition~\ref{defP} produces a holomorphic isomorphism
\begin{gather}\label{beta}
\beta \colon \
\operatorname{Sym}^{n-1}\left(\CJ_2\right) \longrightarrow \CJn .
\end{gather}
We note that
\begin{gather*}
\bigwedge\nolimits^2 \CJ_2 = K^{-1/2}_\Sigma\otimes K^{-1/2}_\Sigma\otimes
K_\Sigma = {\mathcal O}_\Sigma ,
\end{gather*}
and hence the fibers of the vector bundle $\CJ_2$ are equipped with a symplectic structure.
For any $x \in \Sigma$, and for any $v$, $w$ in the fiber $(\CJ_2)_x$ of $\CJ_2$ over $x$, let
\begin{gather}\label{sp0}
\langle v, w\rangle_1 \in \mathbb C
\end{gather}
be this symplectic pairing. Note that we have
\begin{gather}\label{sp}
\langle v, w\rangle_1 = - \langle w, v\rangle_1
\end{gather}
as the pairing $\langle -, -\rangle_1$ is symplectic.

We will show that the above symplectic structure $\langle -, -\rangle_1$ in~\eqref{sp0} on the fibers of $\CJ_2$ produces a bilinear form on the fibers of the symmetric product $\operatorname{Sym}^d(\CJ_2)$
for every $d \geq 1$. For this, take a point $x \in \Sigma$, and take
\begin{gather*}
v_1, \dots, v_d \in \operatorname{Sym}^d(\CJ_2)_x\qquad \text{and}\qquad
w_1, \dots, w_d \in \operatorname{Sym}^d(\CJ_2)_x .
\end{gather*}
Then we have
\begin{gather*}
\underline{v} := v_1\otimes \cdots\otimes v_d \in \operatorname{Sym}^d (\CJ_2 )_x
\qquad \text{and}\qquad \underline{w} := w_1\otimes \cdots\otimes w_d \in \operatorname{Sym}^d (\CJ_2 )_x .
\end{gather*}
Now define the pairing
\begin{gather}\label{dp1}
\langle \underline{v}, \underline{w}\rangle_d := \prod_{i=1}^d \langle v_i, w_i\rangle_1
 \in \mathbb C ,
\end{gather}
where $\langle -, -\rangle_1$ is the pairing in \eqref{sp0}. Note that the pairing $\langle -, -\rangle_d$ in~\eqref{dp1} coincides with $\langle -, -\rangle_1$
in~\eqref{sp0} when $d = 1$. It is straightforward to check that $\langle -, -\rangle_d$ in~\eqref{dp1} produces a~nodegenerate bilinear form on $\operatorname{Sym}^d (\CJ_2 )_x$.

Next note that from \eqref{sp} and \eqref{dp1} we have
\begin{gather*}
\langle \underline{w}, \underline{v}\rangle_d = \prod_{i=1}^d \langle w_i, v_i\rangle_1
 = (-1)^d \prod_{i=1}^d \langle v_i, w_i\rangle_1 = (-1)^d
\langle \underline{v}, \underline{w}\rangle_d .
\end{gather*}
Therefore, the nodegenerate bilinear form $\langle -, -\rangle_d$ on $\operatorname{Sym}^d\left(\CJ_2\right)_x$
is symmetric if $d$ is even and $\langle -, -\rangle_d$ is anti-symmetric if $d$ is odd.

In particular, $\operatorname{Sym}^{n-1}(\CJ_2)$ is equipped with the
orthogonal (respectively, symplectic) form $\langle -, -\rangle_{n-1}$ if $n$ is odd (respectively, even).

For $n$ odd (respectively, even), using the isomorphism $\beta$ in \eqref{beta},
the orthogonal (respectively, symplectic) form
$\langle -, -\rangle_{n-1}$ on $\operatorname{Sym}^{n-1}(\CJ_2)$ produces an
orthogonal (respectively, symplectic) structure on the fibers of $\CJn$.
\end{proof}

\begin{Remark}\label{reml}
The bilinear form on $\CJn$ in Lemma \ref{lemB} coincides with the bilinear form
on $\CJn$ constructed in the proof of Proposition \ref{prop2}.
\end{Remark}

The filtration (see Lemma \ref{kernel1})
\[
0 \longrightarrow K^{1/2}_\Sigma \longrightarrow \CJ_2 = J^1\big(K^{-1/2}_\Sigma\big)
 \longrightarrow J^0\big(K^{-1/2}_\Sigma\big) = K^{-1/2}_\Sigma \longrightarrow 0
\]
of $\CJ_2$ produces a filtration of $\operatorname{Sym}^{n-1}(\CJ_2)$ such that all the successive quotients are of the form $K^{j/2}_\Sigma$, for $1-n \leq j \leq n-1$.

\begin{Lemma}\label{sec} The holomorphic isomorphism $\beta$ in~\eqref{beta} takes the above
filtration of $\operatorname{Sym}^{n-1}(\CJ_2)$ to the filtration of $\CJn$ constructed in~\eqref{e11}.
\end{Lemma}

\begin{proof}This is straightforward. When $\text{genus}(\Sigma) \geq 2$, both these filtrations coincide
with the Harder--Narasimhan filtration of $\operatorname{Sym}^{n-1}(\CJ_2) = \CJn$.
\end{proof}

\subsection[Construction of $B$-opers]{Construction of $\boldsymbol{B}$-opers}

In order to build a converse statement to that of Theorem \ref{prop3}, let $W$ be a holomorphic vector bundle on $\Sigma$ equipped
with a symmetric fiberwise non-degenerate bilinear form
\begin{gather}\label{sv}
{\mathbb S}_W \colon \ \operatorname{Sym}^2(W) \longrightarrow {\mathcal O}_\Sigma .
\end{gather}
Moreover, consider a holomorphic connection $\nabla^W$ on $W$ that preserves the
bilinear form ${\mathbb S}_W$. For an integer $n \geq 2$, we shall let $\omega$ be
an ${\rm Sp}(n, {\mathbb C})$-oper (respectively, ${\rm SO}(n, {\mathbb C})$-oper)
on~$\Sigma$ if~$n$ is even (respectively, odd), which defines a holomorphic connection $\nabla^\omega$
on the holomorphic vector bundle $\CJn$ as defined in~\eqref{jk}. Then, one can show the following.

\begin{Proposition}\label{lem4} The holomorphic connection
\begin{gather*}
(\nabla^\omega\otimes {\rm Id}_W)\oplus \big({\rm Id}_{\CJn} \otimes \nabla^W\big)
\end{gather*}
on $\CJn\otimes W$, induced by $\nabla^\omega$ and $\nabla^W$, produces a generalized $B$-oper on $\Sigma$.
\end{Proposition}

\begin{proof} Consider the filtration of holomorphic subbundles on $\CJn$ given by
\[
0 = A_0 \subset A_1 \subset A_2 \subset \cdots \subset A_{n-1} \subset A_n = \CJn ,
\]
where $A_i$ is the kernel of the natural projection $\CJn \longrightarrow
J^{n-i-1}\big(K^{(1-n)/2}_\Sigma\big)$ (similar to the filtration in~\eqref{e11}). By tensoring with~$W$, the above filtration produces a filtration of holomorphic subbundles of $\CJn\otimes W$ given by
\begin{gather}\label{fo}
0 = A_0\otimes W \subset A_1\otimes W \subset A_2\otimes W \subset \cdots
\subset A_{n-1}\otimes W \subset A_n\otimes W = \CJn\otimes W.
\end{gather}
From \cite[p.~13, equation~(3.4)]{Bi2}, the oper $\omega$ produces a holomorphic differential operator
of order $n$ from $K^{(1-n)/2}_\Sigma$ to $K^{(n+1)/2}_\Sigma$, and from \cite[p.~14, equation~(3.7)]{Bi2} this differential operator produces a~projective structure on $\Sigma$.
Then, as in \eqref{beta}, from \cite[p.~10, Theorem~3.7]{Bi1} this projective structure produces a holomorphic isomorphic between $\operatorname{Sym}^{n-1}(\CJ_2)$ and $\CJn$. Moreover, as observed in the proof of Lemma~\ref{lemB}, the vector bundle $\operatorname{Sym}^{n-1}(\CJ_2)$ is equipped with an
orthogonal (respectively, symplectic) form if the integer $n$ is odd (respectively, even) and thus
 we get a non-degenerate bilinear form ${\mathbb S}_\omega$ on $\CJn$ via the above isomorphism.

The form ${\mathbb S}_\omega$ and the form ${\mathbb S}_W$ in~\eqref{sv} together define a
bilinear form on $\CJn\otimes W$, which we shall denote by
${\mathbb S}_0$. Moreover, one can see that the holomorphic connection
$(\nabla^\omega\otimes {\rm Id}_W)\oplus \big({\rm Id}_{\CJn} \otimes \nabla^W\big)$ on
$\CJn\otimes W$ preserves ${\mathbb S}_0$, because $\nabla^W$ preserves~${\mathbb S}_W$ and $\nabla^\omega$ preserves ${\mathbb S}_\omega$. Then, it is straightforward to check that the triple
\[
\big(\CJn\otimes W, {\mathbb S}_0, (\nabla^\omega\otimes {\rm Id}_W)\oplus \big({\rm Id}_{\CJn} \otimes \nabla^W\big)\big)
\]
and the filtration in \eqref{fo} together define a generalized $B$-oper on $\Sigma$, completing the proof.
\end{proof}

The constructions in Theorem~\ref{prop3} and Proposition~\ref{lem4} are evidently inverses of
each other, and can be further understood in terms of projective structures on the Riemann
surface~$\Sigma$. For this, let ${\mathfrak P}(\Sigma)$ denote the space of all projective
structures on the Riemann surface $\Sigma$. We recall that ${\mathfrak P}(\Sigma)$ is an
affine space for $H^0\big(\Sigma, K^{\otimes 2}_\Sigma\big)$ (see \cite{Gu}).

In the case of classical opers, from \cite[p.~17, Theorem~4.9]{Bi2} and \cite[p.~19, equation~(5.4)]{Bi2},
the space of ${\rm SL}(n, {\mathbb C})$-opers on $\Sigma$ is in bijection with
\[
{\mathfrak P}(\Sigma)\times \left(\bigoplus_{i=3}^n H^0\big(\Sigma, K^{\otimes i}_\Sigma\big)\right) .
\]
Note that the above description of ${\rm SL}(n, {\mathbb C})$-opers allows one to relate them
with the Hitchin base of the ${\rm SL}(n, {\mathbb C})$-Hitchin fibration
introduced in \cite{N2}.

When the integer $n$ is even (respectively, odd) the subclass of ${\rm Sp}(n, {\mathbb C})$-opers
(respectively, ${\rm SO}(n, {\mathbb C})$-opers) corresponds to the subspace
\[
{\mathfrak P}(\Sigma)\times \left(\bigoplus_{i=2}^{[n/2]} H^0\big(\Sigma, K^{\otimes 2i}_\Sigma\big)\right)
 \subset {\mathfrak P}(\Sigma)\times \left(\bigoplus_{i=3}^n H^0\big(\Sigma, K^{\otimes i}_\Sigma\big)\right) .
\]
In the case of generalized $B$-opers, combining Theorem~\ref{prop3} and Proposition~\ref{lem4},
we have the following equivalent result.

\begin{Theorem}\label{thm2}For integers $n \geq 2$, $n \not= 3$ and $r \ge 1$, the space of all generalized $B$-opers of filtration length $n$ and $\operatorname{rank}(E_i/E_{i-1})=r$ is in correspondence with
\[
{\mathcal C}_\Sigma\times {\mathfrak P}(\Sigma)\times \left(\bigoplus_{i=2}^{[n/2]}
H^0\big(\Sigma, K^{\otimes 2i}_\Sigma\big)\right) ,
\]
where ${\mathcal C}_\Sigma$ denotes the space all flat orthogonal bundles of rank $r$ on $\Sigma$, which is independent of~$i$, and ${\mathfrak P}(\Sigma)$ is the space of all projective structures on $\Sigma$.
\end{Theorem}

In Theorem \ref{thm2}, the case of $n = 3$ is excluded because ${\rm SO}(3,{\mathbb C})
 = {\rm Sp}(2,{\mathbb C})/({\mathbb Z}/2{\mathbb Z})$.

The summation in Theorem \ref{thm2} should be clarified~-- the summation is from $2$
to $[n/2]$ in \textit{increasing} order. So $\big(\bigoplus_{i=2}^{[n/2]} H^0\big(\Sigma, K^{\otimes
2i}_\Sigma\big)\big) = 0$ when $n = 2$.

\subsection[Generalized $B$-opers on jet bundles]{Generalized $\boldsymbol{B}$-opers on jet bundles}

Consider now a holomorphic vector bundle $W$ on $\Sigma$ of rank $n$ equipped with a fiberwise
non-degenerate symmetric pairing
\[
\nu \colon \ W\otimes W \longrightarrow K^{1-n}_\Sigma .
\]
As before, using $\nu$ we have a non-degenerate pairing $\langle-, -\rangle$ between
$W$ and $W\otimes K^n_\Sigma$ with values in $K_\Sigma$. The jet bundle $J^{n-1}(W)$ has a filtration
\begin{gather*}
0 = {\mathcal S}_0 \subset {\mathcal S}_1 \subset \cdots \subset {\mathcal S}_{n-1} \subset {\mathcal S}_{n} := J^{n-1}(W) ,
\end{gather*}
which is constructed as done in~\eqref{e11} using the natural projections $J^{n-1}(W) \longrightarrow J^{n-1-i}(W)$. More precisely, ${\mathcal S}_i$ is the kernel of the projection $J^{n-1}(W) \longrightarrow J^{n-1-i}(W)$.

Given a holomorphic differential operator \[
{\mathbb D} \in H^0\big(\Sigma, \operatorname{Diff}^n_\Sigma\big(W, W\otimes K^n_\Sigma\big)\big),
\]
its symbol $\sigma({\mathbb D})$ as in Definition~\ref{symbol} is a holomorphic section
of \[(T\Sigma)^{\otimes n}\otimes
\operatorname{Hom}\big(W, W\otimes K^n_\Sigma\big) = \operatorname{End}(W) .\]

\begin{Definition} We define the differential operator \begin{gather*}
{\mathbb D}^* \in H^0\big(\Sigma, \operatorname{Diff}^n_\Sigma\big(W, W\otimes K^n_\Sigma\big)\big)\end{gather*}
by the equation
$\langle {\mathbb D}(s), t\rangle = \langle s, {\mathbb D}^*(t)\rangle$, where $s$ and $t$
are any locally defined holomorphic sections of the vector bundle~$W$, and
$\langle, - - \rangle$ is the~$K_\Sigma$ valued pairing between~$W$ and~$W\otimes K^n_\Sigma$ mentioned earlier.
\end{Definition}

From this definition one can obtain a correspondence between differential operators $\mathbb D$
as above and generalized $B$-oper structures on the jet bundle $J^{n-1}(W)$.

\begin{Proposition}\label{lem2}The differential operator $\mathbb D$ defines a generalized $B$-oper structure on the vector bundle $J^{n-1}(W)$ equipped with the filtration constructed as in~\eqref{e11} if
\begin{itemize}\itemsep=0pt
\item $\sigma({\mathbb D}) = {\rm Id}_W$, and

\item ${\mathbb D}^* = \mathbb D$.
\end{itemize}
\end{Proposition}

\begin{proof}
In \eqref{dei} we saw that the holomorphic differential operator associated to a generalized
$B$-oper satisfies the above condition ${\mathbb D}^* = \mathbb D$. Also, it
was noted in the proof of Proposition \ref{prop2} that
symbol of the operator is $\text{Id}$ (see \cite[p.~18]{Bi3}).

Consider the subset of $H^0\big(\Sigma, \operatorname{Diff}^n_\Sigma\big(W, W\otimes K^n_\Sigma\big)\big)$ defined by all differential operators whose symbol is $\text{Id}_W$. From \cite[p.~17, Theorem~4.9]{Bi2} and
\cite[p.~19, equation~(5.4)]{Bi2}, this space is
in bijection with
\[
{\mathfrak P}(\Sigma)\times \left(\bigoplus_{i=3}^{n} H^0\big(\Sigma, K^{\otimes i}_\Sigma\big)\right).\] Furthermore, a~holomorphic
differential $D$ lying in this subset satisfies the equation
$D^* = D$ if the component of $D$ in $H^0\big(\Sigma, K^{\otimes (2j+1)}_\Sigma\big)$ vanishes for all
$1 \leq j \leq \floor{(n-1)/2}$, from which the proposition follows.
\end{proof}

\section[Concluding remarks: generalized $B$-opers and Higgs bundles]{Concluding remarks: generalized $\boldsymbol{B}$-opers and Higgs bundles} \label{sec:higgs}

For those having worked with Higgs bundles, introduced in~\cite{N1} as solutions of the so-called Hitchin equations on $\Sigma$, there is proximity between some aspects of $B$-opers and Higgs bundles on $\Sigma$, which are pairs $(E,\Phi)$ where
\begin{itemize}\itemsep=0pt
\item $E$ is a holomorphic vector bundle on $\Sigma$,
\item the Higgs field $\Phi\colon E\rightarrow E \otimes K$, is a holomorphic $K$-valued endomorphism.
\end{itemize}
 To illustrate this, we shall focus on $B$-opers of even rank where $B$ is anti-symmetric, which shall naturally lead to symplectic Higgs bundles. For further details on Higgs bundles, the reader may refer to standard references such as Hitchin~\cite{N1,N2} and Simpson \cite{simpson92,S2,simpson}.
Finally, we should mention that generalized $B$-opers carry particularly interesting properties when the length of the filtration is~2, a case which leads to an intermediate generalization of opers with short filtrations, whose study is carried on in more detail in the third author's PhD thesis~\cite{mengxue}. In what follows, we shall describe some aspects initiating the program. For this, some properties of determinant bundles shall be of much use, and thus we will give a brief description of them first.

We want to find a formula for the determinant bundle for each subbundle in the filtration of a generalized $B$-oper. The proof to the next proposition follows the idea given by Wentworth in \cite[Lemma~4.9]{MR3675465}.

\begin{Proposition}\label{prop:detbundle}
 Suppose $(E,\set{E_i},D)$ is a generalized $B$-oper of filtration length $n$ and each associated graded piece has rank $r = \operatorname{rank}((E_i/E_{i-1})$. Then there are smooth isomorphisms
 \begin{gather}\label{eq:detbundle}
 \det(E_i) \simeq \det(Q)^i \otimes K_\Sigma^{r(ni-i(i+1)/2)} ,
 \end{gather}
 where $Q = E/E_{n-1}$. In particular, when the $B$-oper is a complete flag, we have
\begin{gather*}
\det(E_i) \simeq Q^i \otimes K_\Sigma^{ni - i(i+1)/2} .
\end{gather*}
\end{Proposition}

\begin{proof}The composition of second fundamental forms $S_j(D) \circ \cdots \circ S_{n-1}(D)$ gives an isomorphism for each associated graded piece:
\begin{gather*}
E_j/E_{j-1} \simeq E_n/E_{n-1} \otimes K_\Sigma^{n-j} = Q \otimes K_\Sigma^{n-j},
\qquad \forall\, j = 1, 2, \dots, n.
\end{gather*}
Combining these isomorphisms with the fact that there is a smooth decomposition of~$E$ into associated graded pieces one gets the following isomorphism:
\begin{gather*}
E_i \simeq \big(Q \otimes K_\Sigma^{n-1}\big) \oplus \cdots \oplus \big(Q \otimes K_\Sigma^{n-i}\big).
\end{gather*}
Since $E_i$ has rank $ri$, the determinant bundle of~$E_i$ is given by
\begin{align*}
\bigwedge^{ri} E_i \simeq \bigotimes_{j=1}^i \bigwedge^r \big(Q \otimes K_\Sigma^{n-j}\big)
 \simeq \det(Q)^i \otimes K_\Sigma^{r(ni-i(i+1)/2)} ,
\end{align*}
where we have used the exterior product isomorphisms $\wedge^k (V \oplus W) \simeq \wedge^k V \otimes \wedge^k W$ and \linebreak$\wedge^k (V \otimes W_1) \simeq \wedge^k V \otimes W_1^{\otimes k}$, for $W_1$ a one dimensional vector space.
\end{proof}

 \subsection[Generalized $B$-opers and filtrations of length $n=2$]{Generalized $\boldsymbol{B}$-opers and filtrations of length $\boldsymbol{n=2}$}\label{Higgs}

From their construction, given an anti-symmetric form $B$ and a rank $2r$ holomorphic bundle~$E$, from
Definition \ref{def1} a generalized $B$-oper $(E,{\mathcal F}, D)$ naturally defines a Lagrangian subspace
$E_1|_x\subset E|_x$ for each $x\in \Sigma$ and thus a Lagrangian subbundle $E_1\subset E$.

\begin{Proposition}Let $(E,E_1,D)$ be a generalized $B$-oper of rank $2r$ and filtration length $2$. Then there is a~rank $2r$ Higgs bundle $(E,\Phi)$, where the Higgs field is induced by
\[S_1(D)\colon \ E_1 \longrightarrow E/E_1 \otimes K_\Sigma.\]
In particular, when $r=1$, the Higgs bundle is $\big(K^{1/2} \oplus K^{-1/2},\Phi\big)$.
\end{Proposition}

\begin{proof}
Note that $S_1(D)$ induces an isomorphism $E_1 \to (E/E_1)\otimes K_\Sigma
 = E_1^* \otimes K_\Sigma$, because
the definition of a $B$-filtration gives $E/E_1 \simeq E^*_1$.
Then, $E_1 \oplus (E/E_1) \simeq E_1 \oplus E_1^*$ has the induced Higgs field
\begin{align*}
\Phi = \begin{bmatrix}
0 & 0 \\
S_1(D) & 0
\end{bmatrix} \colon \ E_1 \oplus (E/E_1) \longrightarrow (E_1 \oplus (E/E_1)) \otimes K_\Sigma.
\end{align*}
When $r = 1$ we have $E_1$ is just a line bundle over $\Sigma$. Then the isomorphism given by
$S_1(D)$ implies $E_1^2 \simeq K_\Sigma$. Hence $E_1 \simeq K_\Sigma^{1/2}$
is a choice of theta characteristic, and $E \simeq K_\Sigma^{1/2} \oplus K_\Sigma^{-1/2}$.
\end{proof}

\subsection{Other induced Higgs bundles}

We shall finally consider the appearance of Higgs bundles through generalized $B$-opers with flags of length bigger than 2. In this case, one has the following Higgs bundles constructed in a natural way from generalized $B$-opers.

\begin{Proposition}
On a compact connected Riemann surface of genus $g \ge 2$, a generalized $B$-oper $(E,\set{E_i},D)$ of filtration length $n$ and associated graded rank $r = \operatorname{rank}((E_i/E_{i-1})$ induces a~naturally defined stable Higgs bundle $(E,\Phi)$ given by
\begin{gather*}
E = \bigoplus_{i=1}^{n} E_i/E_{i-1} \qquad \text{and} \qquad \Phi = \begin{bmatrix}
0 & 0 & 0 & 0 & 0\\
S_1(D) & 0 & 0 & 0 & 0\\
0 & S_2(D) & 0 & 0 & 0\\
0 & 0 & \ddots & 0 & 0\\
0 & 0 & 0 & S_{n-1}(D) & 0
\end{bmatrix}\colon \ E \longrightarrow E \otimes K_\Sigma,
\end{gather*}
where $S_i(D)\colon E_i/E_{i-1} \longrightarrow E_{i+1}/E_i \otimes K_\Sigma$ are the second
fundamental forms with respect to~$D$.
\end{Proposition}

\begin{proof}We need to check that $(E,\Phi)$ is a stable Higgs bundle. From the decomposition of $E$, we see that $E/E_{n-1}$ is the only $\Phi$-invariant subbundle, because $\Phi(E/E_{n-1}) = 0 \subset E/E_{n-1} \otimes K$. So according to the slope stability definition, we know $(E,\Phi)$ is stable if $\mu(E) - \mu(E/E_{n-1}) > 0$. Write $Q = E/E_{n-1}$ and apply equation~\eqref{eq:detbundle} to compute the degree of $E = E_n$:
 \begin{align*}
 \mu(E) - \mu(E/E_{n-1}) &= \frac{\deg(E_n)}{\operatorname{rank}(E_n)} - \frac{\deg(Q)}{\operatorname{rank}(Q)}
 = \frac{n\deg(Q) + rn(n-1)(g-1)}{rn} - \frac{\deg(Q)}{r} \\
 &= (n-1)(g-1).
 \end{align*}
 Thus for $n > 1$ and $g > 1$, the Higgs bundle $(E,\Phi)$ is always stable.
\end{proof}

\begin{Remark}
One should note that all of the Higgs bundles induced by $B$-opers we constructed in
this section lie in the nilpotent cone in the Hitchin fibration, because the characteristic
polynomial of $\Phi$ is $\det(xI - \Phi) = x^n$.
\end{Remark}

\subsubsection{The rank 4 case}
We shall conclude here with some further comments on the lowest possible rank of generalized $B$-opers which have a symmetric bilinear form~$B$. For this, consider now a $B$-\textit{filtration} of a holomorphic
vector bundle~$E$ of rank~4 on $\Sigma$, which is given by an increasing filtration of holomorphic subbundles
\begin{gather}\label{eq44}
0 = E_0 \subsetneq
E_1 \subsetneq E_2 \subsetneq E_{3} \subsetneq E_{4} = E
\end{gather}
for which $E^\perp_i = E_{4-i}$ for all $1 \leq i \leq 3$. Then, from Definition~\ref{def1}, we have that a generalized $B$-oper is a triple $(E,{\mathcal F}, D)$, where $\mathcal F$ is a $B$-filtration as in \eqref{eq44} and $D$ is a $B$-connection on~$E$, such that $D(E_i) \subset E_{i+1}\otimes K_\Sigma$ for all $1 \leq i \leq 3$, and
the homomorphisms
\begin{gather*}
S_1(D)\colon \ E_1 \longrightarrow (E_{2}/E_1)\otimes K_\Sigma, \\
S_2(D) \colon \ E_2/E_{1} \longrightarrow (E_{3}/E_2)\otimes K_\Sigma,\\
S_3(D) \colon \ E_3/E_{2} \longrightarrow (E/E_3)\otimes K_\Sigma
\end{gather*}
are isomorphisms. Through the bilinear form $B$ there is an isomorphism $
E^*_1 \stackrel{\sim}{\longrightarrow} E /E_{3}
$, and in this case one has that $E_{2} = E^\perp_2$ and thus we obtain again a Lagrangian sub-fibration. Hence, a rank $4$ generalized $B$-oper $(E,E_1,D)$ induces the following naturally defined Higgs bundle $(P,\Phi)$, where $P=E_3$ and $\Phi\colon E_3\rightarrow E_3\otimes K$ is induced by
\begin{gather*}
S_1 \oplus S_2 \oplus S_3 \colon \ E_1 \oplus E_2/E_1 \oplus E_3/E_2
 \longrightarrow (E_2/E_1 \otimes K) \oplus (E_3/E_2 \otimes K) \oplus (E/E_3 \otimes K) ,
\end{gather*}
since $E_1 \oplus E_2/E_1 \oplus E_3/E_2 \simeq
E_3$ and $(E_2/E_1 \otimes K) \oplus (E_3/E_2 \otimes K) \oplus (E/E_3 \otimes K) \simeq E_3 \otimes K$.

\subsection*{Acknowledgments}

The authors are grateful for the hospitality and support of the Simons Center for Geometry and Physics' program {\it Geometry $\&$ Physics of Hitchin Systems} co-organized by L.~Anderson and L.P.~Schaposnik, where this project started~-- in particular, they would like to thank Andy Sanders for his seminar at the program~\cite{andy}, which inspired some of the ideas in this paper. Finally, we would like to thank Brian Collier, Aaron Fenyes, Nigel Hitchin, Steve Rayan and Sebastian Schulz for comments on a draft of the manuscript. The authors are also thankful for all the very useful comments from the two anonymous referees, whose suggestions greatly improved the manuscript. LS is grateful for Motohico Mulase's support and encouragement over the years. She is partially supported by the NSF grant DMS-1509693, the NSF CAREER Award DMS-1749013, and by the Alexander von Humboldt Foundation. This material is also based upon work supported by the National Science Foundation under Grant No. DMS-1440140 while the author was in residence at the Mathematical Sciences Research Institute in Berkeley, California, during the Fall 2019 semester. IB is supported by a J.C.~Bose Fellowship.

\pdfbookmark[1]{References}{ref}
\LastPageEnding

\end{document}